\documentclass[preprint,12pt,3p]{elsarticle}
\usepackage{verbatim}
\usepackage[active,new,old]{correct} 
\usepackage{amsmath,amsthm,amsfonts,amssymb}
\usepackage{algorithm}
\usepackage{algpseudocode}
\usepackage{mathdots}
\usepackage{cases}
\usepackage[hidelinks]{hyperref}
\usepackage{graphics}
\usepackage{graphicx}
\usepackage{caption}
\usepackage{subcaption}
\usepackage{multicol}
\usepackage[normalem]{ulem}
\usepackage{xcolor,sectsty}
\definecolor{astral}{RGB}{46,116,181}
\subsectionfont{\color{astral}}
\sectionfont{\color{astral}}
\newtheorem{theorem}{Theorem}[section]
\newtheorem{lemma}[theorem]{Lemma}
\newtheorem{corollary}[theorem]{Corollary}

\newtheorem{definition}[theorem]{Definition}
\newtheorem{example}[theorem]{Example}

\definecolor{darkslategray}{rgb}{0.18, 0.31, 0.31}
\definecolor{warmblack}{rgb}{0.0, 0.26, 0.26}
\definecolor{thrdfc}{rgb}{0.36, 0.54, 0.66}
\definecolor{bole}{rgb}{0.55, 0.71, 0.0}

\journal{...}

\newcommand{\mc}{\mathcal}
\newcommand{\mb}{\mathbb}
\newcommand{\dg}{{\dagger}}
\newcommand{\m}{*_{M}}
\newcommand{\n}{*_{N}}

\newcommand{\1}{*_{1}}
\newcommand{\2}{*_{2}}

\begin{document}

\begin{frontmatter}

\title{ \textcolor{warmblack}{\bf Numerical range for weighted Moore-Penrose inverse of  tensor}}

\author{Aaisha Be$^{\dag, a}$, Vaibhav Shekhar$^{*, b}$, Debasisha Mishra{$^{\dag, c, \mu}$}}

\address{$^{\dag}$Department of Mathematics,\\
National Institute of Technology Raipur,\\
Raipur, Chhattisgarh, India.\\
\vspace{0.15cm}
$^{^*}$Department of Mathematics,\\
Indian Institute of Technology Delhi,\\
New Delhi-110016, India.\\

$^{^\mu}$Corresponding Author\\

\textit{E-mail$^a$}: \texttt{aaishasaeed7\symbol{'100}gmail.com}\\
\textit{E-mail$^b$}: \texttt{vaibhav\symbol{'100}maths.iitd.ac.in}\\
\textit{E-mail$^c$}: \texttt{dmishra@nitrr.ac.in}}

                      \begin{abstract}
\textcolor{warmblack}{This article first introduces the notion of weighted singular value decomposition (WSVD) of a tensor via the Einstein product. The WSVD is then used to compute the weighted Moore-Penrose inverse of an arbitrary-order tensor. We then define the notions of weighted normal tensor for an even-order square tensor and weighted tensor norm. Finally, we apply these to study the theory of numerical range for the weighted Moore-Penrose inverse of an even-order square tensor and exploit its several properties. We also obtain a few new results in the matrix setting that generalizes some of the existing results as particular cases.}
\end{abstract}

\begin{keyword}
Tensor, Einstein product, Numerical range,  Numerical radius, Weighted Moore-Penrose inverse.\\
\vspace{0.2cm}
{\it Mathematics Subject Classification}: 15A69, 15A60.
\end{keyword}

\end{frontmatter}

\section{Introduction}\label{sec1}
The terms numerical range and numerical radius have drawn significant attention from the researchers in the last few decades in the field of matrix and operator theory \cite{bonsall1971, bonsall1973, halmos1967, horn1991}. These have been widely studied because of their applications in many areas, such as 
numerical analysis and differential equations \cite{axelsson1994, cheng1999, eiermann1993, fiedler1995, goldberg1982, kirkland2001, maroulas2002, psarrakos2002}. The numerical range (or the field of values) of a square matrix $A\in \mb{C}^{n\times n}$ is a subset of complex numbers defined as: 
                  \begin{equation}\label{eq:nrmtrx}
      W(A)=\{\left\langle Ax,x\right\rangle:\;x\in\mb{C}^{n},\|x\|=1\},
\end{equation}
where $\left\langle x,y\right\rangle=y^{*}x$ for $x,y\in \mb{C}^{n}$ and $\|x\|=\left\langle x,x\right\rangle^{1/2}$. And, the numerical radius of the matrix $A$ is defined as:
            \begin{equation}\label{def:nrd matrx}
              w(A)=\textnormal{max}\{|z|:\;z\in  W(A)\}.
\end{equation}
One of the main reasons for emphasizing the numerical range concept is its many attractive properties. For example, $W(A)$ is a convex subset of $\mb{C}$ (known as Toeplitz-Hausdorff theorem \cite{horn1991}). Further, the numerical range of a matrix contains its spectrum (or the set of all eigenvalues). The numerical radius 
is frequently employed as a more reliable indicator of the rate of convergence of iterative methods than the spectral radius \cite{axelsson1994, eiermann1993}. In 2016, Ke {\it et al.} \cite{ke2016}  introduced tensor numerical ranges using tensor inner products and tensor norms via the $k$-mode product, which may not be convex in general (see Example 1, \cite{ke2016}). In 2021, Pakmanesh and Afshin \cite{pakmanesh2021} continued the same study for even-order tensors and proved the convexity for the numerical range of an even-order tensor.
In the same year, Rout {\it et al.} \cite{nirmal} introduced tensor numerical ranges using tensor inner products and tensor norms via the Einstein product. The authors \cite{nirmal}  studied several fundamental notions of tensor numerical ranges, such as unitary invariance, spectral containment, and convexity. 
Furthermore, they developed an algorithm to plot the boundary of the numerical range of a tensor, which helps to design faster algorithms for the calculations of its eigenvalues. 
To understand tensor numerical ranges, we first recall some basic facts about tensors. Tensors are generalizations of scalars (that have no index), vectors (that have precisely one index), and matrices (that have precisely two indices) to an arbitrary number of indices. A tensor is represented as a multidimensional array. An $N^{th}$-order tensor is defined as:  $$ \mc{A}=(a_{i_{1}i_{2}\ldots i_{N}})\in \mb{C}^{I_{1}\times \cdots \times I_{N}};~1\leq i_k \leq I_k~\textnormal{for each}~k=1,2,...,N, $$ where each mode $I_k$ is a natural number and the notation $a_{i_{1}\ldots i_{N}}$ represents an $(i_{1},\ldots,i_{N})^{th}$ element of $\mc{A}$. For simplicity, $I_{1}\times I_{2}\times \cdots \times I_{N}$ is also written as $I_{1\ldots N}$. The transpose of a tensor may not be unique, which is recalled here.
Let $\mc{A}\in \mb{C}^{I_{1}\times I_{2}\times \cdots \times I_{M}}$ be a tensor and let $\pi$ be a permutation in $S_{M}$ except the identity permutation, where $S_{M}$ represents the permutation group over the set $\{1, 2, \ldots, M\}$. Then, the $\pi$-transpose of the tensor $\mc{A}$ is defined as
           \begin{equation}
           \mc{A}^{T_{\pi}}=(a_{i_{\pi(1)}i_{\pi(2)}\ldots i_{\pi(M)}})\in \mb{C}^{I_{\pi(1)}\times I_{\pi(2)}\times \cdots \times I_{\pi(M)}}.
\end{equation}
Thus, there are $M!-1$ possible transposes associated with a tensor $\mc{A}\in \mb{C}^{I_{1}\times I_{2}\times \cdots \times I_{M}}$. 
In particular, for $\mc{A}=(a_{i_{1}i_{2}\ldots i_{M}j_{1}j_{2} 
\ldots j_{N}})\in \mb{C}^{I_{1}\times I_{2}\times \cdots \times I_{M}\times J_{1}\times J_{2}\times \cdots \times J_{N}}$ and $\pi\in S_{M+N}$ such that 
$\mc{A}^{T_{\pi}}=(b_{j_{1}j_{2}\ldots j_{N}i_{1}i_{2}\ldots i_{M}})=(a_{i_{1}i_{2}\ldots i_{M}j_{1}j_{2}\ldots j_{N}})\in \mb{C}^{J_{1}\times J_{2}\times \cdots \times J_{N}\times I_{1}\times I_{2}\times \cdots \times I_{M}}$, then it is simply written as $\mc{A}^{T}$. Similarly, 
 the conjugate transpose of $\mc{A}$ is denoted by $\mc{A}^{H}$ and defined by $\mc{A}^{H}=(c_{j_{1}j_{2}\ldots j_{N}i_{1}i_{2}\ldots i_{M}})=(\overline{a}_{i_{1}i_{2}\ldots i_{M}j_{1}j_{2}\ldots j_{N}})\in \mb{C}^{J_{1}\times J_{2}\times \cdots \times J_{N}\times I_{1}\times I_{2}\times \cdots \times I_{M}}$, where bar denotes the complex conjugate of a number. Furthermore, if $\mc{A}\in \mb{C}^{I_{1}\times I_{2}\times \cdots \times I_{M}}$, then $\mc{A}^{T}=(a_{1i_{1}i_{2}\ldots i_{M}})\in \mb{C}^{1\times I_{1}\times I_{2}\times \cdots \times I_{M}}$.
There are two ways to define a square tensor. One when each mode is of equal size, i.e., $n\times n\times \cdots \times n$ and another when the first $N$ modes are repeated in the same order, i.e., $I_{1}\times I_{2}\times \cdots \times I_{N}\times I_{1}\times I_{2}\times \cdots \times I_{N}$. Ke {\it et al.} \cite{ke2016} extended the notion of the numerical range of a matrix for the former type of square tensors.  
Further, Rout {\it et al.} \cite{nirmal} extended the numerical range for the latter type of square tensors. They also obtained a few properties of tensor
numerical range of the Moore–Penrose inverse via the Einstein product. We aim to study these properties of tensor numerical range for the weighted Moore-Penrose inverse of a tensor. For this purpose, we recall the Einstein product below.\par
The Einstein product \cite{einstein2007} $\mc{A}\n\mc{B} \in \mb{C}^{I_{1\ldots M} \times J_{1\ldots L} }$ of tensors $\mc{A} \in \mb{C}^{I_{1 \ldots M} \times K_{1 \ldots N} }$ and $\mc{B} \in \mb{C}^{K_{1\ldots N} \times J_{1\ldots L} }$ is defined by the operation $\n$ via
            \begin{equation*}\label{Eins}
            (\mc{A}\n\mc{B})_{i_{1}\ldots i_{M}j_{1}\ldots j_{L}}
           =\displaystyle\sum_{k_{1},\ldots, k_{N}}a_{{i_{1}\ldots i_{M}}{k_{1}\ldots k_{N}}}b_{{k_{1}\ldots k_{N}}{j_{1}\ldots j_{L}}}.
\end{equation*}
The associative law for the Einstein product holds. In the above formula, if $\mc{B} \in \mb{C}^{K_{1\ldots N}}$, then $\mc{A}\n\mc{B} \in \mb{C}^{I_{1 \ldots M}}$ and 
             \begin{equation*}
            (\mc{A}\n\mc{B})_{i_{1}\ldots i_{M}} = \displaystyle\sum_{k_{1},\ldots, k_{N}}
            a_{{i_{1}\ldots i_{M}}{k_{1}\ldots k_{N}}}b_{{k_{1}\ldots k_{N}}}.
\end{equation*}
This product is used in the study of the theory of relativity \cite{einstein2007} and in the area of continuum mechanics \cite{lai2009}. Let $A\in \mb{R}^{m \times n}$ and $B\in \mb{R}^{n \times l}$. Then, the Einstein product $\1$  reduces to the standard matrix multiplication as
$$(A\1B)_{ij}= \displaystyle\sum_{k=1}^{n} a_{ik}b_{kj}.$$
 \noindent We refer to \cite{huang2020} for further advantages of studying the theory of tensors via the Einstein product. 
 In 2019, Liang and Zheng \cite{liang2019}  introduced the notion of the eigenvalues of an even-order square tensor via the Einstein product as follows.
\begin{definition}[Definition 2.3, \cite{liang2019}]\label{egndfn}\leavevmode\\
Let $\mc{A}\in \mb{C}^{I_{1 \ldots N}\times I_{1 \ldots N}}$. Then, a complex number $\lambda$ is called an eigenvalue of $\mc{A}$ if there exists a nonzero tensor $\mc{X}\in \mb{C}^{I_{1 \ldots N}}$ such that  
              \begin{equation} 
             \mc{A}\n \mc{X}=\lambda \mc{X}.
\end{equation} 
     The tensor $\mc{X}$ is called an eigentensor with respect to $\lambda$.
\end{definition}
\noindent The set of all the eigenvalues of $\mc{A}$ is denoted by $\sigma(\mc{A})$. The spectral radius of the tensor $\mc{A}$ is  denoted by $\rho(\mc{A})$ and defined by
$\rho(\mc{A})=\max\{|\lambda|: \lambda \text{ is an eigenvalue of } \mc{A}\}.$ Furthermore, the positive square roots of eigenvalues of $\mc{A}^{H}\n\mc{A}$ are called the {\it singular values} of $\mc{A}$. The maximum singular value of $\mc{A}$ is called the {\it spectral norm} \cite{ma2019} of the tensor $\mc{A}$. 
{In 2013, Brazell {\it et al.} \cite{brazell2013}  first introduced the notion of the inverse of a tensor via the Einstein product. For $\mc{A}\in\mb{C}^{I_{1\ldots N}\times I_{1\ldots N}}$, if there exists a tensor $\mc{X}\in\mb{C}^{I_{1\ldots N}\times I_{1\ldots N}}$ such that $\mc{A}\n\mc{X}=\mc{I}=\mc{X}\n\mc{A}$, then the tensor $\mc{X}$ is called the {\it inverse} of the tensor $\mc{A}$ and it is denoted by $\mc{A}^{-1}$.}
In 2016, Sun {\it et al.} \cite{sun2016} formally introduced  a generalized inverse called the {\it Moore-Penrose inverse} of an even-order tensor via the Einstein product. The authors \cite{sun2016} then used the Moore-Penrose inverse to find the minimum-norm least-squares solution of some multilinear systems. Panigrahy and Mishra \cite{panigrahy2018}, Stanimirovi\'c {\it et al.} \cite{stanimirovic2020},  and  Liang and Zheng \cite{liang2019} independently improved the definition of the Moore-Penrose inverse of an even-order tensor to a tensor of any order via the same product.  In 2017, Ji and Wei \cite{ji2017} defined the weighted Moore-Penrose inverse for an even-order square tensor, and then in 2020, Behera {\it et al.} \cite{behera2020} extended the definition to an arbitrary-order tensor. The definition of the weighted Moore-Penrose inverse of an arbitrary-order tensor and one result are recalled here.
          \begin{definition}[Definition 8, \cite{behera2020}]\label{defwmpi}\leavevmode\\
          Let $\mc{A} \in \mb{C}^{I_{1 \ldots M} \times J_{1 \ldots N}}$, and $\mc{M} \in \mb{C}^{I_{1 \ldots M} \times I_{1 \ldots M}}$, $\mc{N} \in \mb{C}^{J_{1 \ldots N} \times J_{1 \ldots N}}$ be two Hermitian positive definite tensors. Then, the tensor $\mc{X} \in \mb{C}^{J_{1 \ldots N} \times I_{1 \ldots M}}$ is called the weighted Moore-Penrose inverse of $\mc{A}$ if it satisfies the following four tensor equations:
    \begin{eqnarray}
       \mc{A}\n\mc{X}\m\mc{A} &=& \mc{A};\label{wmpeq1}\\
       \mc{X}\m\mc{A}\n\mc{X} &=& \mc{X};\label{wmpeq2}\\
      (\mc{M}\m \mc{A}\n\mc{X})^{H} &=& \mc{M}\m\mc{A}\n\mc{X};\label{wmpeq3}\\
      (\mc{N}\n\mc{X}\m\mc{A})^{H} &=& \mc{N}\n\mc{X}\m\mc{A}.\label{wmpeq4}
\end{eqnarray}
It is denoted by $\mc{A}_{\mc{M},\mc{N}}^{\dg}$.
In particular, if $\mc{M}$ and $\mc{N}$ are the identity tensors, then $\mc{A}_{\mc{M},\mc{N}}^{\dg}=\mc{A}^{\dg}$, the Moore-Penrose inverse \cite{panigrahy2018} of $\mc{A}.$  
\end{definition}

                  \begin{theorem}[Theorem 4, \cite{behera2020}]\label{thmwmpi}\leavevmode\\
Let $\mc{A} \in \mb{C}^{I_{1 \ldots M} \times J_{1 \ldots N}}$, and $\mc{M} \in \mb{C}^{I_{1 \ldots M} \times I_{1 \ldots M}}$, $\mc{N} \in \mb{C}^{J_{1 \ldots N} \times J_{1 \ldots N}}$ be two Hermitian positive definite tensors. If $\tilde{\mc{A}}=\mc{M}^{1/2}\m\mc{A}\n\mc{N}^{-1/2}$, then $$\mc{A}_{\mc{M}, \mc{N}}^{\dg}=\mc{N}^{-1/2}\n\tilde{\mc{A}}^{\dg}\m\mc{M}^{1/2}.$$ 
\end{theorem}

\noindent This article aims to introduce the notions of WSVD of an arbitrary-order tensor, weighted normal tensor, and weighted tensor norm and to establish their various properties. Some of these are utilized to investigate a few properties of the numerical range for the weighted Moore-Penrose inverse of an even-order square tensor.
The rest of this article is structured as follows to accomplish our objectives. In Section \ref{sec:preliminaries}, we recall some preliminaries. Then, we provide the WSVD and some of its applications in Section \ref{sec:WSVD}. Section \ref{sec:wnormaltnsr} defines the weighted normal tensor and discusses its several features. In Section \ref{sec:wtnsrnorm}, we talk about the weighted tensor norm. Finally, we utilize all these notions to collect some properties of the numerical range, which examine different relations between the numerical range of a tensor and its weighted Moore-Penrose inverse in Section \ref{sec:nrwmpi}.

\section{Preliminaries}\label{sec:preliminaries}
For two tensors $\mc{X},\mc{Y}\in\mb{C}^{I_{1\ldots N}}$, an inner product $\left\langle\mc{X},\mc{Y}\right\rangle$ is defined as $\left\langle\mc{X},\mc{Y}\right\rangle=\mc{Y}^{H}\n\mc{X}$ and a norm induced by this inner product as $\|\mc{X}\|=\left\langle\mc{X},\mc{X}\right\rangle^{1/2}$. A tensor $\mc{X}\in\mb{C}^{I_{1\ldots N}}$ is called a {\it unit tensor} if $\|\mc{X}\|=1$. First, we recall the definition of the numerical range and some results from \cite{nirmal}.
             \begin{definition}[Definition 2.1, \cite{nirmal}]\label{def:nrtn}\leavevmode\\
Let $\mc{A}\in \mb{C}^{I_{1\ldots N}\times I_{1\ldots N}}$. Then, the numerical range of $\mc{A}$ is denoted by $W(\mc{A})$
and defined by
          \begin{equation}\label{nrtnsrdfn}
          W(\mc{A})
             =\{\left\langle \mc{A}\n\mc{X},\mc{X}\right\rangle:\;\mc{X}\text{ is a unit tensor in }\mb{C}^{I_{1\ldots N}}\}.
        \end{equation}
\end{definition}
\noindent With some elementary calculations, it can be shown that
           \begin{equation}\label{eqn:nraltdfn}
           W(\mc{A})=\left\{\frac{\left\langle \mc{A}\n\mc{X},\mc{X}\right\rangle}{\|\mc{X}\|^{2}}:\;\mc{O}\neq\mc{X}\in\mb{C}^{I_{1\ldots N}}\right\},
\end{equation} where $\mc{O}$ is the zero
tensor having all the entries are zero.
\noindent
The numerical radius for the same tensor is defined as:
           \begin{equation}\label{def:nrd tnsr}
              w(\mc{A})=\textnormal{max}\{|z|:\;z\in  W(\mc{A})\}.
\end{equation}
 Note that, in the above Definition \ref{def:nrtn} when $N=1$, it coincides with the numerical range of a matrix defined in \eqref{eq:nrmtrx}.
          \begin{theorem}[Theorem 5.1, \cite{nirmal}]\label{thm:normaltn}\leavevmode\\
Let $\mc{A}\in \mb{C}^{I_{1\ldots N}\times I_{1\ldots N}}$. Then, $\mc{A}$ is normal (resp. Hermitian) if and only if $\mc{A}^{\dg}$ is normal (resp. Hermitian). 
\end{theorem}

\noindent We now recall the definition of the weighted conjugate transpose of a tensor proposed by Behera {\it et al.} \cite{behera2020} 
             \begin{definition}[Definition 9, \cite{behera2020}]\label{defwct}\leavevmode\\
             Let $\mc{A} \in \mb{C}^{I_{1 \ldots M} \times J_{1 \ldots N}}$. If $\mc{M} \in \mb{C}^{I_{1 \ldots M} \times I_{1 \ldots M}}$ and $\mc{N} \in \mb{C}^{J_{1 \ldots N} \times J_{1 \ldots N}}$ are two  Hermitian positive definite tensors, then the tensor 
          \begin{equation}\label{eqnwct}
         \mc{A}_{\mc{M}\mc{N}}^{\#}= \mc{N}^{-1}\n \mc{A}^H\m\mc{M}
\end{equation}
is called the weighted conjugate transpose of $\mc{A}$.
\end{definition}
The following definitions of tensor reshaping are derived from \cite{stanimirovic2020}.
In 2020, Stanimirovi\'c {\it et al.} \cite{stanimirovic2020}  defined a bijective map `rsh' from the tensor space ${\mathbb{C}}^{M_{1\cdots m}\times N_{1\cdots n}}$ into the matrix space ${\mathbb{C}}^{(M_{1}\cdot M_{2}\cdot \cdots \cdot M_{m})\times (N_{1}\cdot N_{2}\cdot \cdots\cdot N_{n})}$ as follows:
        \begin{definition}[Definition 3.1, \cite{stanimirovic2020}]\label{def:reshape}\leavevmode\\
The reshaping operation transforms an arbitrary tensor $\mc{A}\in \mb{C}^{M_{1 \ldots m}\times N_{1\ldots n}}$ into the matrix $A\in \mb{C}^{\bf{m}\times \bf{n}}$ using the \it{Matlab} function \textnormal{reshape} as follows:  
$$\textnormal{rsh}(\mc{A})=A=reshape(\mc{A},\bf{m}, \bf{n}),$$ where ${\bf{m}}=M_1\cdot M_2\cdot\cdots\cdot M_m$ and ${\bf{n}}=N_1\cdot N_2\cdot\cdots\cdot N_n.$ 
The inverse reshaping is the mapping defined by 
$$ \textnormal{rsh}^{-1}(A)=\mc{A}=reshape(A, M_1,\ldots,M_m,N_1,\ldots,N_n).$$
\end{definition}
\noindent Further, rshrank($\mc{A}$)=rank(rsh($A$)), is defined as the {\it{rank}} of the tensor $\mc{A}$, in the same article \cite{stanimirovic2020}.
For the Einstein product of two tensors, reshape map satisfies the following property given as Lemma 3.1 in \cite{stanimirovic2020}.
                  \begin{lemma}[Lemma 3.1, \cite{stanimirovic2020}]\label{lem:rsh EP prop}\leavevmode\\
Let $\mc{A}\in \mb{C}^{M_{1 \ldots m}\times N_{1\ldots n}}$ and $\mc{B}\in \mb{C}^{N_{1\ldots n}\times P_{1 \ldots p}}$ be given tensors. Then, $$\textnormal{rsh}(\mc{A}\n\mc{B})=\textnormal{rsh}(\mc{A})\textnormal{rsh}(\mc{B})=AB,$$ where $\textnormal{rsh}(\mc{A})=A\in \mb{C}^{\bf{m}\times \bf{n}}$, $\textnormal{rsh}(\mc{B})=B\in \mb{C}^{\bf{n}\times \bf{p}}$, ${\bf{m}}=M_1\cdot M_2\cdot\cdots\cdot M_m$, ${\bf{n}}=N_1\cdot N_2\cdot\cdots\cdot N_n$, and ${\bf{p}}=P_1\cdot P_2\cdot\cdots\cdot P_p.$
\end{lemma}

\section{Weighted singular value decomposition}\label{sec:WSVD} 
This section contains some of the main results of this article. We prove that any tensor can be decomposed into the tensor Einstein product of three special tensors. We call it the WSVD of the given tensor as it generalizes the notion of the WSVD of a matrix \cite{loan1976}. After that, we derive a formula to compute the weighted Moore-Penrose inverse of a given arbitrary-order tensor. For applications of the WSVD of a matrix, we refer \cite{jozi2018, kyrchei2017, naeem2021, sergienko2015} and references therein. In particular, the WSVD is widely used in solving weighted least squares solutions \cite{loan1976, wang2004}.
We now propose the WSVD of an arbitrary-order tensor using the reshaping operation that generalizes Theorem 3.17 of \cite{brazell2013}, Lemma 3.1 of \cite{sun2016}, Theorem 3.2 of \cite{liang2019}, and Lemma 2 of \cite{behera2020}. 
           \begin{theorem}\label{thm:WSVD}
Let $\mc{A}\in \mb{C}^{I_{1\ldots M}\times J_{1\ldots N}}$ with $\textnormal{rshrank}(\mathcal{A})=r$, and $\mc{M}\in \mb{C}^{I_{1\ldots M}\times I_{1\ldots M}}$, $\mc{N}\in \mb{C}^{J_{1\ldots N}\times J_{1\ldots N}}$ be two Hermitian positive definite tensors. Then, there exist tensors $\mc{U}\in\mb{C}^{I_{1\ldots M}\times I_{1\ldots M}} $ and $\mc{V}\in\mb{C}^{J_{1\ldots N}\times J_{1\ldots N}}$ satisfying $\mc{U}^H\m\mc{M}\m\mc{U}=\mc{I}_1$ and $\mc{V}^H\n \mc{N}^{-1} \n \mc{V}=\mc{I}_2$, where $\mc{I}_1\in\mb{C}^{I_{1\ldots M}\times I_{1\ldots M}} $ and $\mc{I}_2\in\mb{C}^{J_{1\ldots N}\times J_{1\ldots N}}$ are the identity tensors, such that 
           \begin{equation}\label{eqn:tnsrWSVD}
         \mc{A}=\mc{U}\m \mc{S}\n \mc{V}^H 
\end{equation}
in which the tensor $\mc{S}=(\mc{S}_{i_{1}\hdots i_{M} j_{1}\hdots j_{N}})\in \mb{R}^{I_{1\ldots M}\times J_{1\ldots N}}$ is defined by 
          \begin{equation}\label{eqn:sigma WSVD}
          \mc{S}_{i_{1}\cdots i_{M}j_{1}\cdots j_{N}}=\begin{cases}
          \mu_{IJ}>0,&\text{if } I=J\in\{1,~2,~\ldots,~r\},\\
         0,&\text{otherwise},
         \end{cases}
 \end{equation}
 where $I=\big[i_{1}+\displaystyle\sum_{s=2}^{M}(i_{s}-1)\prod_{u=1}^{s-1}I_{u}\big]$ and $J=\big[j_{1}+\displaystyle\sum_{t=2}^{N}(j_{t}-1)\prod_{v=1}^{t-1}J_{v}\big]$.
 
 \end{theorem}
           \begin{proof}
Let $A=\textnormal{rsh}(\mc{A})\in  \mb{C}^{m\times n} $, $M=\textnormal{rsh}(\mc{M})\in \mb{C}^{m\times m}$, and $N=\textnormal{rsh}(\mc{M}) \in \mb{C}^{n\times n}$ be the reshapings of the tensors $\mc{A}$, $\mc{M}$, and $\mc{N}$, respectively, where
          \begin{equation}\label{eqn:product of modes}
           m=I_{1}\cdot I_{2}\cdot \cdots \cdot I_{M} ~\textnormal{and}~ n=J_{1}\cdot J_{2}\cdot \cdots \cdot J_{N}. 
\end{equation}
The WSVD of the matrix $A$ with respect to the weights $M$ and $N$ is
            \begin{equation}\label{eqn:matWSVD}
            A=USV^*,
\end{equation}
where $U\in \mb{C}^{m\times m}$ and $V\in \mb{C}^{n\times n}$ following $U^*MU=I_1$ and $V^*N^{-1}V=I_2$, $I_1\in \mb{C}^{m\times m}$, $I_2\in \mb{C}^{n\times n}$ are the identity matrices, and $S=\textnormal{diag}(\mu_1, \mu_2,\hdots,\mu_r,0,\hdots,0)\in \mb{R}^{m\times n}$, $\mu_1 \geq \mu_2 \geq...\geq \mu_r >0$ are the nonzero $(M,N)$ singular values of $A$. 
Since $\textnormal{rsh}$ is a bijection, taking the inverse map $\textnormal{rsh}^{-1}$ on both sides of  \eqref{eqn:matWSVD}, we obtain 
            \begin{align*}
            \textnormal{rsh}^{-1}(A)&=\textnormal{rsh}^{-1}(US V^*)\\
             &=\textnormal{rsh}^{-1}(U)\m\textnormal{rsh}^{-1}(S)\n\textnormal{rsh}^{-1}(V^*)\\
             &= \textnormal{rsh}^{-1}(U)\m\textnormal{rsh}^{-1}(S)\n[\textnormal{rsh}^{-1}(V)]^H,
\end{align*}
which implies that
             \begin{align*}
             \mc{A}&= \mc{U}\m \mc{S}\n \mc{V}^H.
\end{align*}
Now, we have  $$U^*MU=I_1~\textnormal{and}~V^*N^{-1}V=I_2.$$ Applying the reverse map $\textnormal{rsh}^{-1}$ on both sides in the last two equalities, we get $$\mc{U}^H\m\mc{M}\m\mc{U}=\mc{I}_1 ~\textnormal{and} ~\mc{V}^H\n \mc{N}^{-1} \n \mc{V}=\mc{I}_2.$$ 
From $\mc{S}=\textnormal{rsh}^{-1}(S)$,  
\begin{equation*}
          \mc{S}_{i_{1}\cdots i_{M}j_{1}\cdots j_{N}}=\begin{cases}
          \mu_{IJ}>0,&\text{if } I=J\in\{1,~2,~\ldots,~r\},\\
         0,&\text{otherwise}.
         \end{cases}
 \end{equation*} 
 This completes the proof.
\end{proof}
We call \eqref{eqn:tnsrWSVD} as the WSVD of the tensor $\mc{A}$ and ${\mu_{IJ}}^,s$ as the $(\mc{M},\mc{N})$ singular values of $\mc{A}.$
 We next present an algorithm for computing the WSVD.\\
            \begin{algorithm}[h!]
             \caption{\texttt{Computing WSVD of a tensor} }\label{alg1}
             \begin{algorithmic}[1]
\Require { \texttt{Positive integers $M, N, I_1,...,I_M, J_1,...,J_N, m,n$ such that $m$ and $n$ satisfy \eqref{eqn:product of modes}. $A\in \mb{C}^{m\times n}$, and $M\in \mb{C}^{m\times m}$, $N\in \mb{C}^{n\times n}$ two Hermitian positive definite matrices. }}
\State \texttt{Compute the WSVD of $A$, $$A=USV^*,$$
 where $U\in \mb{C}^{m\times m}$ and $V\in\mb{C}^{n\times n}$ and $S\in\mb{C}^{m\times n}$ is a diagonal matrix with $(M,N)$ singular values of $A$ on the main diagonal.}
\State \texttt{Perform the reshaping operations 
$$\textnormal{rsh}^{-1}(U)= \mc{U}\in \mb{C}^{I_{1\ldots M}\times I_{1\ldots M}},~\textnormal{rsh}^{-1}(V^*)= \mc{V}^H\in \mb{C}^{J_{1\ldots N}\times J_{1\ldots N}},~\textnormal{rsh}^{-1}(S)= \mc{S}\in \mb{C}^{I_{1\ldots M}\times J_{1\ldots N}}.$$}
\State \texttt{Compute the output
$$\mc{A}= \mc{U}\m \mc{S}\n \mc{V}^H.$$}
\end{algorithmic}
\end{algorithm}
\newpage
The following example demonstrates the above algorithm.  
             \begin{example}\label{exm:WSVD}
              Consider the tensor $\mc{A}\in  \mb{C}^{I_1\times I_2\times J_1}=\mb{C}^{2\times2\times2}$ and the two weights $\mc{M}\in  \mb{C}^{I_1\times I_2\times I_1\times I_2}=\mb{C}^{2\times2\times2\times 2}$, $\mc{N}\in  \mb{C}^{J_1\times J_1}=\mb{C}^{2\times2}$ such that

               \begin{center}
               \begin{tabular}{c c  | c c }
\hline
  \multicolumn{2}{c}{$\mc{A}(:,:,1)$} & \multicolumn{2}{c}{$\mc{A}(:,:,2)$} \\
\hline
1 & 0 & 0 & 1 \\
0 & 0 & 0 & 0  \\
\hline
\end{tabular},
\end{center}
               \begin{center}
               \begin{tabular}{cc|cc|cc|cc}
\hline
\multicolumn{2}{c}{$\mc{M}(:,:,1,1)$} & \multicolumn{2}{c}{$\mc{M}(:,:,2,1)$} &
\multicolumn{2}{c}{$\mc{M}(:,:,1,2)$} &
\multicolumn{2}{c}{$\mc{M}(:,:,2,2)$} \\
\hline
    1 & 0 & 0 & 0 & 0 & 1 & 0 & 0 \\
    0 & 0 & 1 & 0 & 0 & 0 & 0 & 4 \\
\hline 
\end{tabular},
\end{center}
and
                 \begin{center}
                 \begin{tabular}{c c}
 \hline
   \multicolumn{2}{c}{$\mathcal{N}(:,:)$} \\
 \hline
 4 & 0  \\
 0 & 1  \\
 \hline
\end{tabular}.
\end{center}
On reshaping these tensors $\mc{A}, \mc{M},$ and $\mc{N}$, we obtain the matrices $A, M,$ and $N$, respectively, as follows:
                 \begin{center}
 $A=\begin{bmatrix}1 & 0\\
0 & 0 \\
0 & 1\\
0 & 0
\end{bmatrix},$
$M=\begin{bmatrix}1&0 &0&0\\
0&1 &0&0\\
0&0&1&0\\
0&0&0&4\end{bmatrix},$ and 
$N=\begin{bmatrix}4 & 0\\
0 & 1
\end{bmatrix}.$
\end{center} 
Now, computing the WSVD of the matrix $A$ with respect to the weights $M$ and $N$, we get
$A=USV^*$, where  
                  \begin{center}
$U=\begin{bmatrix}
0&1&0&0\\
0&0&1&0\\
1&0&0&0\\
0&0&0&\frac{1}{2}\end{bmatrix},$
$S=\begin{bmatrix}1 & 0\\
0 & \frac{1}{2} \\
0 & 0\\
0 & 0
\end{bmatrix},$ and 
$V=\begin{bmatrix}
0 & 2\\
1 & 0
\end{bmatrix}.$ 
\end{center} 
On applying the inverse reshape function on the matrices $U,S,$ and $V$, we get the tensors $\mc{U}\in\mb{C}^{I_1\times I_2\times I_1\times I_2}=\mb{C}^{2\times2\times2\times 2}, \mc{S}\in  \mb{C}^{I_1\times I_2\times J_1}=\mb{C}^{2\times2\times2},$ and $\mc{V}\in  \mb{C}^{J_1\times J_1}=\mb{C}^{2\times2}$, respectively as 
                 \begin{center}
                 \begin{tabular}{cc|cc|cc|cc}
\hline
\multicolumn{2}{c}{$\mc{U}(:,:,1,1)$} & \multicolumn{2}{c}{$\mc{U}(:,:,2,1)$} &
\multicolumn{2}{c}{$\mc{U}(:,:,1,2)$} &
\multicolumn{2}{c}{$\mc{U}(:,:,2,2)$} \\
\hline
    0 & 1 & 1 & 0 & 0 & 0 & 0 & 0 \\
    0 & 0 & 0 & 0 & 1 & 0 & 0 & $\frac{1}{2}$ \\
\hline 
\end{tabular}, 
\end{center}
                  \begin{center}
                  \begin{tabular}{c c  | c c }
\hline
  \multicolumn{2}{c}{$\mc{S}(:,:,1)$} & \multicolumn{2}{c}{$\mc{S}(:,:,2)$} \\
\hline
1 & 0 & 0 & 0 \\
0 & 0 & $\frac{1}{2}$ & 0  \\
\hline
\end{tabular},
\end{center}
and
                \begin{center}
                \begin{tabular}{c c}
 \hline
   \multicolumn{2}{c}{$\mathcal{V}(:,:)$} \\
 \hline
 0 & 2  \\
 1 & 0  \\
 \hline
\end{tabular}.
\end{center}
Then, $\mc{U}\2\mc{S}$ becomes   
                 \begin{center}
                 \begin{tabular}{c c  | c c }
\hline
  \multicolumn{2}{c}{$\mc{U}\2\mc{S}(:,:,1)$} & \multicolumn{2}{c}{$\mc{U}\2\mc{S}(:,:,2)$} \\
\hline
0 & 1 & $\frac{1}{2}$  & 0 \\
0 & 0 & 0 & 0  \\
\hline
\end{tabular}.
\end{center} 
Therefore, $\mc{U}\2\mc{S}\1 \mc{V}^H$ is 
                  \begin{center}
                  \begin{tabular}{c c  | c c }
\hline
  \multicolumn{2}{c}{$\mc{U}\2\mc{S}\1 \mc{V}^H(:,:,1)$} & \multicolumn{2}{c}{$\mc{U}\2\mc{S}\1 \mc{V}^H(:,:,2)$} \\
\hline
1 & 0 & 0 & 1 \\
0 & 0 & 0 & 0  \\
\hline
\end{tabular}, 
\end{center}
which is same as the tensor $\mc{A},$ i.e., $\mc{A}=\mc{U}\2\mc{S}\1 \mc{V}^H.$ 
\end{example}
Next, we provide a result to compute the weighted Moore-Penrose inverse of a tensor $\mc{A}$ via the WSVD, $\mc{A}= \mc{U}\m \mc{S}\n \mc{V}^H$.
                  \begin{theorem}\label{thm:WMPI via WSVD}
Let $\mc{A}\in \mb{C}^{I_{1\ldots M}\times J_{1\ldots N}}$ with $\textnormal{rshrank}(\mathcal{A})=r$, and $\mc{M}\in \mb{C}^{I_{1\ldots M}\times I_{1\ldots M}}$, $\mc{N}\in \mb{C}^{J_{1\ldots N}\times J_{1\ldots N}}$ be two Hermitian positive definite tensors. If $\mc{A}= \mc{U}\m \mc{S}\n \mc{V}^H$ is the WSVD of the tensor $\mc{A}$, then 
                 \begin{equation}\label{eqn: WMPI via WSVD}
                \mc{A}_{\mc{M},\mc{N}}^{\dg}=\mc{N}^{-1}\n\mc{V}\n {\mc{S}}^{\dg}\m \mc{U}^H\m \mc{M},
\end{equation} where $\mc{S}^{\dg}=(\mc{S}^{\dg}_{j_{1}\hdots j_{N} i_{1}\hdots i_{M}})\in \mb{R}^{J_{1\ldots N}\times I_{1\ldots M}}$ is defined by 
               \begin{equation}\label{eqn: sigma dagger WSVD}
            \mc{S}^{\dg}_{j_{1}\cdots j_{N}i_{1}\cdots i_{M}}=\begin{cases}
            \mc{S}^{-1}_{i_1 \cdots i_Mj_1 \cdots j_N  },&\text{if }\mc{S}_{ i_1 \cdots i_M j_1 \cdots j_N} \neq 0,\\
             0,&\text{otherwise}.
             \end{cases}
\end{equation} 
\end{theorem}
                \begin{proof}
It can be easily proved by Definition \ref{defwmpi}.
\end{proof}
 Here, we present an algorithm for computing the weighted Moore-Penrose inverse via the WSVD.
 \newpage
                \begin{algorithm}[h!]
\caption{\texttt{Computing weighted Moore-Penrose inverse of a tensor}}\label{alg2}
                \begin{algorithmic}[1]
\Require { \texttt{Positive integers $M, N, I_1,...,I_M, J_1,...,J_N, m,n$ such that $m$ and $n$ satisfy \eqref{eqn:product of modes}. $A\in \mb{C}^{m\times n}$, and $M\in \mb{C}^{m\times m}$, $N\in \mb{C}^{n\times n}$ two Hermitian positive definite matrices.}}
\State \texttt{Compute the WSVD of $A$, $$A=USV^*,$$
 where $U\in \mb{C}^{m\times m}$ and $V\in\mb{C}^{n\times n}$ and $S\in\mb{C}^{m\times n}$ is a diagonal matrix with $(M,N)$ singular values of $A$ on the main diagonal.} 
\State \texttt {Perform the reshaping operations 
$$\textnormal{rsh}^{-1}(U)= \mc{U}\in \mb{C}^{I_{1\ldots M}\times I_{1\ldots M}},~\textnormal{rsh}^{-1}(V^*)= \mc{V}^H\in \mb{C}^{J_{1\ldots N}\times J_{1\ldots N}},~\textnormal{rsh}^{-1}(S)= \mc{S}\in \mb{C}^{I_{1\ldots M}\times J_{1\ldots N}}.$$}
\State \texttt{Compute the output
$$\mc{A}_{\mc{M},\mc{N}}^{\dg}= \mc{N}^{-1}\n\mc{V}\n \mc{S}^{\dg}\m\mc{U}^H\m \mc{M}.$$}
\end{algorithmic}
\end{algorithm}
The next example verifies the above algorithm.  
                \begin{example}\label{exm:WMPI via WSVD}
Consider the same tensors $\mc{A}, \mc{M},$ and $\mc{N}$ as given in Example \ref{exm:WSVD}. By the same argument, we have the tensors $\mc{U}, \mc{S},$ and $\mc{V}.$ Using \eqref{eqn: sigma dagger WSVD}, we get $\mc{S}^{\dg}\in  \mb{C}^{J_1\times I_1\times I_2}=\mb{C}^{2\times2\times2}$ as 
               \begin{center}
               \begin{tabular}{c c  | c c }
\hline
  \multicolumn{2}{c}{$\mc{S}^{\dg}(:,:,1)$} & 
  \multicolumn{2}{c}{$\mc{S}^{\dg}(:,:,2)$} \\
\hline
1 & 0 & 0 & 0 \\
0 & 2 & 0 & 0  \\
\hline
\end{tabular}. 
\end{center}
The conjugate transpose $\mc{U}^H\in\mb{C}^{I_1\times I_2\times I_1\times I_2}=\mb{C}^{2\times2\times2\times 2}$ of the tensor $\mc{U}$ is
                \begin{center}
                \begin{tabular}{cc|cc|cc|cc}
\hline
\multicolumn{2}{c}{$\mc{U}^H(:,:,1,1)$} &
\multicolumn{2}{c}{$\mc{U}^H(:,:,2,1)$} &
\multicolumn{2}{c}{$\mc{U}^H(:,:,1,2)$} &
\multicolumn{2}{c}{$\mc{U}^H(:,:,2,2)$} \\
\hline
    0 & 0 & 0 & 1 & 1 & 0 & 0 & 0 \\
    1 & 0 & 0 & 0 & 0 & 0 & 0 & $\frac{1}{2}$ \\
\hline 
\end{tabular}. 
\end{center}
Then, we obtain the tensors $\mc{B}=\mc{N}^{-1}\1\mc{V}$, $\mc{C}=\mc{B}\1\mc{S}^{\dg}$, $\mc{D}=\mc{C}\2\mc{U}^H$, and $\mc{E}=\mc{D}\2\mc{M}$ as follows: 
                \begin{center}
                \begin{tabular}{c c}
 \hline
   \multicolumn{2}{c}{$\mathcal{B}(:,:)$} \\
 \hline
 0 & $\frac{1}{2}$  \\
 1 & 0  \\
 \hline
 \end{tabular},
\end{center}
                \begin{center}
                \begin{tabular}{c c  | c c }
\hline
  \multicolumn{2}{c}{$\mc{C}(:,:,1)$} & 
  \multicolumn{2}{c}{$\mc{C}(:,:,2)$} \\
\hline
0 & 1 & 0 & 0 \\
1 & 0 & 0 & 0  \\
\hline
\end{tabular}, 
\end{center}
                \begin{center}
                \begin{tabular}{c c  | c c }
\hline
  \multicolumn{2}{c}{$\mc{D}(:,:,1)$} & 
  \multicolumn{2}{c}{$\mc{D}(:,:,2)$} \\
\hline
1 & 0 & 0 & 0 \\
0 & 0 & 1 & 0  \\
\hline
\end{tabular}, 
\end{center} 
and 
                 \begin{center}
                 \begin{tabular}{c c  | c c }
\hline
  \multicolumn{2}{c}{$\mc{E}(:,:,1)$} & 
  \multicolumn{2}{c}{$\mc{E}(:,:,2)$} \\
\hline
1 & 0 & 0 & 0 \\
0 & 0 & 1 & 0  \\
\hline
\end{tabular}.
\end{center}  
It is clear that the tensor $\mc{E}=\mc{N}^{-1}\1\mc{V}\1\mc{S}^{\dg}\2\mc{U}^H\2\mc{M}$, which is the weighted Moore-Penrose inverse $\mc{A}_{\mc{M},\mc{N}}^{\dagger}$ of $\mc{A}.$ 
\end{example}
The next result can be easily verified using the definition of the weighted Moore-Penrose inverse of a tensor. 
                  \begin{lemma}\label{lem1}
Let $\mc{A}\in \mb{C}^{I_{1\ldots M}\times J_{1\ldots N}}$, and  $\mc{M}\in \mb{C}^{I_{1\ldots M}\times I_{1\ldots M}}$, $\mc{N}\in \mb{C}^{J_{1\ldots N}\times J_{1\ldots N}}$ be two Hermitian positive definite tensors. If for any tensor  $\mc{B}\in \mb{C}^{I_{1\ldots M}\times J_{1\ldots N}}$, $\mc{A}=\mc{U}\m \mc{B}\n\mc{V}^H$, where $\mc{U}\in \mb{C}^{I_{1\ldots M}\times I_{1\ldots M}}$ and $\mc{V}\in \mb{C}^{J_{1\ldots N}\times J_{1\ldots N}}$ satisfying $\mc{U}^H\m \mc{M} \m \mc{U}=\mc{I}_1$ and $\mc{V}^H\n \mc{N}^{-1} \n \mc{V}=\mc{I}_2$, where $\mc{I}_1\in\mb{C}^{I_{1\ldots M}\times I_{1\ldots M}} $ and $\mc{I}_2\in\mb{C}^{J_{1\ldots N}\times J_{1\ldots N}}$ are the identity tensors, then $\mc{A}_{\mc{M},\mc{N}}^{\dagger}=\mc{N}^{-1} \n \mc{V}\n\mc{B}^{\dagger}\m \mc{U}^H\m \mc{M}. $  
\end{lemma}
The above result reduces to the following corollary in the matrix case.
                   \begin{corollary}\label{cor:lem1.1}
Let $A\in\mb{C}^{m\times n}$, and $M\in\mb{C}^{m\times m}$, $N\in\mb{C}^{n\times n}$ be two Hermitian positive definite matrices. If for any matrix  $B\in\mb{C}^{m\times n}$, $A=UBV^*$, where $U\in\mb{C}^{m\times m}$ and $V\in\mb{C}^{n\times n}$ satisfying $U^*MU=I_1$ and $V^*N^{-1}V=I_2$, where $I_1\in \mb{C}^{m\times m}$ and $I_2\in \mb{C}^{n\times n}$ are the identity matrices, then $A_{M,N}^{\dg}=N^{-1}VB^{\dg}U^*M.$  
\end{corollary}
 In the following theorem, we present a representation of the weighted Moore-Penrose inverse $\mc{A}_{\mc{M},\mc{N}}^{\dagger}$ of $\mc{A}.$ 
                    \begin{theorem}\label{thm:limitWMPI}
  Let $\mc{A} \in \mb{C}^{I_{1 \ldots M} \times J_{1 \ldots N}}$. If $\mc{M} \in \mb{C}^{I_{1 \ldots M} \times I_{1 \ldots M}}$ and $\mc{N} \in \mb{C}^{J_{1 \ldots N} \times J_{1 \ldots N}}$ are two  Hermitian positive definite tensors, then $$
\mc{A}_{\mc{M},\mc{N}}^{\dagger}=\lim_{\lambda \to 0}[(\lambda\mc{I}_2+\mc{A}_{\mc{M}\mc{N}}^{\#}\m\mc{A})^{-1}\n\mc{A}_{\mc{M}\mc{N}}^{\#}],$$ where $\mc{A}_{\mc{M}\mc{N}}^{\#}$ is the weighted conjugate transpose of $\mc{A}$, $\lambda\in \mb{R}^+$, $\mb{R}^+$ denotes the set of all positive real numbers, and $\mc{I}_2\in\mb{C}^{J_{1\ldots N}\times J_{1\ldots N}}$ is the identity tensor.
\end{theorem}
                    \begin{proof}
By Theorem \ref{thm:WSVD} and Theorem \ref{thm:WMPI via WSVD}, we have  $\mc{A}= \mc{U}\m \mc{S}\n \mc{V}^H$ and $\mc{A}_{\mc{M},\mc{N}}^{\dg}=\mc{N}^{-1}\n\mc{V}\n {\mc{S}}^{\dg}\m \mc{U}^H\m \mc{M},$ where \begin{equation*}
          \mc{S}_{i_{1}\cdots i_{M}j_{1}\cdots j_{N}}=\begin{cases}
          \mu_{IJ}>0,&\text{if } I=J\in\{1,~2,~\ldots,~r\},\\
         0,&\text{otherwise},
         \end{cases}
 \end{equation*} 
 and
 \begin{equation*}
            \mc{S}^{\dg}_{j_{1}\cdots j_{N}i_{1}\cdots i_{M}}=\begin{cases}
            \mc{S}^{-1}_{i_1 \cdots i_Mj_1 \cdots j_N  },&\text{if }\mc{S}_{ i_1 \cdots i_M j_1 \cdots j_N} \neq 0,\\
             0,&\text{otherwise},
             \end{cases}
\end{equation*} 
 where $I=\big[i_{1}+\displaystyle\sum_{s=2}^{M}(i_{s}-1)\prod_{u=1}^{s-1}I_{u}\big]$ and $J=\big[j_{1}+\displaystyle\sum_{t=2}^{N}(j_{t}-1)\prod_{v=1}^{t-1}J_{v}\big]$.
 Now, using $\mc{U}^H\m\mc{M}\m\mc{U}=\mc{I}_1~ \textnormal{and}~ \mc{V}^H\n \mc{N}^{-1}\n\mc{V}=\mc{I}_2$, we obtain 
                   \small{\begin{align*}
                 \mc{A}_{\mc{M}\mc{N}}^{\#}\m \mc{A}&= \mc{N}^{-1}\n \mc{A}^H\m \mc{M}\m \mc{A}\\
                &=\mc{N}^{-1}\n \mc{V}\n \mc{S}^H\m \mc{U}^H\m\mc{M}\m \mc{U}\m \mc{S}\n \mc{V}^H\\
               &= (\mc{V}^H)^{-1}\n \mc{S}^H\m \mc{S}\n\mc{V}^H.
\end{align*}} 
Therefore, we get
                 \begin{align*}
               \lambda\mc{I}_2+\mc{A}_{\mc{M}\mc{N}}^{\#}\m\mc{A}&=(\mc{V}^H)^{-1}\n (\lambda\mc{I}_2+\mc{S}^H\m \mc{S})\n \mc{V}^H 
\end{align*} 
and 
                \small{\begin{align*}
              (\lambda\mc{I}_2+\mc{A}_{\mc{M}\mc{N}}^{\#}\m\mc{A})^{-1}&\n \mc{A}_{\mc{M}\mc{N}}^{\#}\\&=(\mc{V}^H)^{-1}\n (\lambda\mc{I}_2+\mc{S}^H\m \mc{S})^{-1}\n \mc{V}^H\n\mc{N}^{-1}\n \mc{V}\n \mc{S}^H\m \mc{U}^H\m\mc{M}\\
              &= \mc{N}^{-1}\n\mc{V} \n (\lambda\mc{I}_2+\mc{S}^H\m \mc{S})^{-1}\n\mc{S}^H\m \mc{U}^H\m\mc{M}.
\end{align*}}
 Now, 
               \begin{equation*}
               \mc{S}^H_{j_{1}\cdots j_{N}i_{1}\cdots i_{M}}=\overline{\mc{S}}_{i_{1}\cdots i_{M}j_{1}\cdots j_{N}}=\begin{cases}
               \mu_{IJ}>0,&\text{if } J=I\in\{1,~2,~\ldots,~r\},\\
               0,&\text{otherwise},
\end{cases}
\end{equation*}
and 
               \begin{equation*}
              (\mc{S}^H \m \mc{S})_{j_{1}\cdots j_{N}k_{1}\cdots k_{N}}=\begin{cases}
              {\mu_{IJ}}^2,&\text{if } J=K\in\{1,~2,~\ldots,~r\},\\
             0,&\text{otherwise}.
\end{cases}
\end{equation*} 
So, we obtain 
               \begin{equation*}
               (\lambda\mc{I}_2+\mc{S}^H\m \mc{S})^{-1}_{j_{1}\cdots j_{N}k_{1}\cdots k_{N}}=\begin{cases}
              \frac{1}{\lambda+{\mu_{IJ}}^2} ,&\text{if } J=K\in\{1,~2,~\ldots,~r\},\\
             \frac{1}{\lambda},&\text{if } J=K\notin\{1,~2,~\ldots,~r\},\\ 0,&\text{otherwise},
\end{cases}
\end{equation*}
and 
               \begin{equation*}
              ((\lambda\mc{I}_2+\mc{S}^H\m \mc{S})^{-1}\n\mc{S}^H)_{j_{1}\cdots j_{N}i_{1}\cdots i_{M}}=\begin{cases}
             \frac{\mu_{IJ}}{\lambda+{\mu_{IJ}}^2},&\text{if } J=I\in\{1,~2,~\ldots,~r\},\\
           0,&\text{otherwise}.
\end{cases}
\end{equation*}
The last equation implies that $\displaystyle \lim_{\lambda \to 0}((\lambda\mc{I}_2+\mc{S}^H\m \mc{S})^{-1}\n\mc{S}^H)=\mc{S}^{\dg}.$ Thus, we get 
                 \begin{align*}
  \lim_{\lambda \to 0}[(\lambda\mc{I}_2+\mc{A}_{\mc{M}\mc{N}}^{\#}\m\mc{A})^{-1}\n\mc{A}_{\mc{M}\mc{N}}^{\#}]&= \lim_{\lambda \to 0} [\mc{N}^{-1}\n\mc{V} \n (\lambda\mc{I}_2+\mc{S}^H\m \mc{S})^{-1}\n\mc{S}^H\m \mc{U}^H\m\mc{M}]\\
  &=\mc{N}^{-1}\n\mc{V} \n \mc{S}^{\dg}\m \mc{U}^H\m\mc{M}\\
  &=\mc{A}_{\mc{M},\mc{N}}^{\dagger}. 
\end{align*}   
\end{proof}
Same as Theorem \ref{thm:limitWMPI}, we can derive the next representation for $\mc{A}_{\mc{M},\mc{N}}^{\dagger}.$ 
            \begin{theorem}\label{thm:limitWMPI 2}
Let $\mc{A}\in \mb{C}^{I_{1\ldots M}\times J_{1\ldots N}}$. If $\mc{M}\in \mb{C}^{I_{1\ldots M}\times I_{1\ldots M}}$ and $\mc{N}\in \mb{C}^{J_{1\ldots N}\times J_{1\ldots N}}$ are two Hermitian positive definite tensors, then 
$$\mc{A}_{\mc{M},\mc{N}}^{\dagger}=\lim_{\lambda \to 0}[\mc{A}_{\mc{M}\mc{N}}^{\#}\m(\lambda\mc{I}_1+\mc{A}\n\mc{A}_{\mc{M}\mc{N}}^{\#})^{-1}],$$ where $\mc{A}_{\mc{M}\mc{N}}^{\#}$ is the weighted conjugate transpose of $\mc{A}$, $\lambda\in \mb{R}^+$, and $\mc{I}_1\in\mb{C}^{I_{1\ldots M}\times I_{1\ldots M}}$ is the identity tensor.
\end{theorem}
By Proposition 2.4 of \cite{sun2016} and Definition \ref{defwmpi}, we can prove Lemma \ref{lem2}.
                 \begin{lemma}\label{lem2}
Let $\mc{B}\in \mb{C}^{I_{1\ldots N}\times I_{1\ldots N}}$ be an invertible tensor and let a block tensor $\mc{A}$ be defined by
                  \begin{equation*}
                 \mc{A}=\begin{bmatrix}\mc{B}&\mc{O}\\\mc{O}&\mc{O}\end{bmatrix},
\end{equation*} 
where $\mc{O}\in \mb{C}^{I_{1\ldots N}\times I_{1\ldots N}}$ is the zero tensor. Let $\mc{M}, \mc{N} \in \mb{C}^{2I_{1}\times\ldots\times 2I_ {N}\times 2I_{1}\times\ldots\times 2I_ {N}} $ be two diagonal tensors whose diagonal entries are positive.
Then,
                 \begin{equation*}
                \mc{A}_{\mc{M},\mc{N}}^{\dg}=\begin{bmatrix}\mc{B}^{-1}&\mc{O}\\
                \mc{O}&\mc{O}\end{bmatrix}.
\end{equation*}
\end{lemma}
The above result reduces to the following corollary when we consider identity tensor as weights.
                 \begin{corollary}\label{cor:lem2.1}
Let $\mc{B}\in \mb{C}^{I_{1\ldots N}\times I_{1\ldots N}}$ be an invertible tensor and let a block tensor $\mc{A}$ be defined by
                \begin{equation*}
                \mc{A}=\begin{bmatrix}\mc{B}&\mc{O}\\\mc{O}&\mc{O}\end{bmatrix},
\end{equation*}
where $\mc{O}\in \mb{C}^{I_{1\ldots N}\times I_{1\ldots N}}$ is the zero tensor. Then, 
                \begin{equation*}
                \mc{A}^{\dg}=\begin{bmatrix}\mc{B}^{-1}&\mc{O}\\\mc{O}&\mc{O}\end{bmatrix}.
\end{equation*}
\end{corollary}

\section{Weighted normal tensor}\label{sec:wnormaltnsr}
In this section, we first introduce the notions of weighted self-conjugate and weighted normal tensor, and exploit their various properties. Further, we show that Theorem \ref{thm:normaltn} does not hold if we replace $\mc{A}^{\dg}$ by $\mc{A}_{\mc{M},\mc{N}}^{\dagger}$. 
                 \begin{definition}
Let $\mc{A}\in \mb{C}^{I_{1\ldots N}\times I_{1\ldots N}}$ be an even-order square tensor, and $\mc{N}\in \mb{C}^{I_{1\ldots N}\times I_{1\ldots N}}$ be a Hermitian positive definite tensor. Then, the tensor $\mc{A}$ is called weighted self-conjugate tensor if $\mc{A}_{\mc{N}\mc{N}}^{\#}=\mc{A}$, i.e., $$ \mc{N}^{-1}\n \mc{A}^H\n\mc{N}=\mc{A}.$$   
\end{definition}
                 \begin{definition}
Let $\mc{A}\in \mb{C}^{I_{1\ldots N}\times I_{1\ldots N}}$ be an even-order square tensor, and $\mc{N}\in \mb{C}^{I_{1\ldots N}\times I_{1\ldots N}}$ be a Hermitian positive definite tensor. Then, the tensor $\mc{A}$ is called the weighted normal tensor if $$\mc{A}_{\mc{N}\mc{N}}^{\#}\n\mc{A}=\mc{A}\n\mc{A}_{\mc{N}\mc{N}}^{\#}.$$ In particular, if $\mc{N}$ is the identity tensor, then the tensor $\mc{A}$ becomes a normal tensor.
\end{definition}
If $N=1$, then $A\in \mathbb{C}^{I_1\times I_1}$ is a square matrix, and $N\in \mathbb{C}^{I_1\times I_1}$ is a Hermitian positive definite matrix. Then, $A$ is called  \textit{weighted normal matrix} if $AA^{\#}=A^{\#}A$, where $A^{\#}=N^{-1}A^*N$ is the weighted conjugate transpose of $A.$
Here, we provide an example which shows that $\mc{A}_{\mc{M},\mc{N}}^{\dagger}$ is not necessarily normal (or Hermitian) tensor even if $\mc{A}$ is normal (or Hermitian) or $\mc{A}^{\dagger}$ is normal (or Hermitian).
                  \begin{example}\label{exnwmp}
Let $A=\begin{pmatrix}
4&0\\
0&0\\
\end{pmatrix}\in \mathbb{C}^{2\times2}$. If $M=\begin{pmatrix}
1&0\\
0&2\\
\end{pmatrix}$ and $N=\begin{pmatrix}
2&1\\
1&2\\
\end{pmatrix}$ are two Hermitian positive definite matrices in $\mathbb{C}^{2\times2}$, then $A^{\dagger}= \begin{pmatrix}
\frac{1}{4}&0\\
0&0\\
\end{pmatrix}$ and $A_{M,N}^{\dagger}=\begin{pmatrix}
\frac{1}{4}&0\\
-\frac{1}{8}&0\\
\end{pmatrix}.$ Thus, we have 
                     \begin{eqnarray*}
AA^*=A^*A&=&     
                     \begin{pmatrix}
16&0\\
0&0\\
\end{pmatrix},\\
A^{\dagger}(A^{\dagger})^*=(A^{\dagger})^*A^{\dagger}&=&\begin{pmatrix}
\frac{1}{16}&0\\
0&0\\
\end{pmatrix},\\
A_{M,N}^{\dagger}(A_{M,N}^{\dagger})^*&=& 
                    \begin{pmatrix}
\frac{1}{16}&-\frac{1}{32}\\
-\frac{1}{32}&\frac{1}{64}\\
\end{pmatrix},\\
(A_{M,N}^{\dagger})^*A_{M,N}^{\dagger}&=& 
                    \begin{pmatrix}
\frac{5}{64}&0\\
0&0\\
\end{pmatrix}.\end{eqnarray*}
Clearly, $A_{M,N}^{\dagger}(A_{M,N}^{\dagger})^* \neq (A_{M,N}^{\dagger})^*A_{M,N}^{\dagger}$, i.e., $A_{M,N}^{\dagger}$ is not normal, also it is not Hermitian. 
\end{example}
The following result gives a necessary and sufficient condition for the weighted Moore-Penrose inverse of a tensor to be weighted normal.
                     \begin{theorem}\label{thmwnr}
Let $\mc{A}\in \mb{C}^{I_{1\ldots N}\times I_{1\ldots N}}$ be an even-order square tensor. If $\mc{N}\in \mb{C}^{I_{1\ldots N}\times I_{1\ldots N}}$ is a Hermitian positive definite tensor, then
                     \begin{enumerate}[(i)]
\item $\mc{A}=\mc{A}_{\mc{N}\mc{N}}^{\#}$ if and only if $\tilde{\mc{A}}= \tilde{\mc{A}}^H$;
\item $\mc{A}\n\mc{A}_{\mc{N}\mc{N}}^{\#}= \mc{A}_{\mc{N}\mc{N}}^{\#}\n\mc{A}$ if and only if $ \tilde{\mc{A}}\n\tilde{\mc{A}}^H=\tilde{\mc{A}}^H\n\tilde{\mc{A}}$;  
\item $\mc{A}_{\mc{N},\mc{N}}^{\dagger}\n(\mc{A}_{\mc{N},\mc{N}}^{\dagger})_{\mc{N}\mc{N}}^{\#}=(\mc{A}_{\mc{N},\mc{N}}^{\dagger})_{\mc{N}\mc{N}}^{\#}\n\mc{A}_{\mc{N},\mc{N}}^{\dagger}$ if and only if $\tilde{\mc{A}}^{\dagger}\n(\tilde{\mc{A}}^{\dagger})^H=(\tilde{\mc{A}}^{\dagger})^H\n\tilde{\mc{A}}^{\dagger}$;
\item $\mc{A}\n\mc{A}_{\mc{N}\mc{N}}^{\#}= \mc{A}_{\mc{N}\mc{N}}^{\#}\n\mc{A}$ if and only if $ \mc{A}_{\mc{N},\mc{N}}^{\dagger}\n(\mc{A}_{\mc{N},\mc{N}}^{\dagger})_{\mc{N}\mc{N}}^{\#}=(\mc{A}_{\mc{N},\mc{N}}^{\dagger})_{\mc{N}\mc{N}}^{\#}\n\mc{A}_{\mc{N},\mc{N}}^{\dagger}$, 
\end{enumerate}
where $\mc{A}^H~\textnormal{and}~\mc{A}_{\mc{N}\mc{N}}^{\#}$ are the conjugate transpose and the weighted conjugate transpose of the tensor $\mc{A}$, respectively and $\tilde{\mc{A}}=\mc{N}^{1/2}\n\mc{A}\n\mc{N}^{-1/2}.$  
\end{theorem}
In the matrix setting, the above result is as follows.
                    \begin{corollary}\label{cor:thmwnr1}
Let $A\in \mathbb{C}^{n\times n}$. If $N\in  \mathbb{C}^{n\times n}$ is a Hermitian positive definite matrix, then
                     \begin{enumerate}[(i)]
\item $A=A^{\#}$ if and only if $\tilde{A}=(\tilde{A})^*;$
    \item  $AA^{\#}=A^{\#}A$ if and only if $\tilde{A}(\tilde{A})^*=(\tilde{A})^*\tilde{A};$
    \item $A_{N,N}^{\dagger}(A_{N,N}^{\dagger})^{\#}=(A_{N,N}^{\dagger})^{\#}A_{N,N}^{\dagger}$ if and only if $\tilde{A}^{\dagger}(\tilde{A}^{\dagger})^*=(\tilde{A}^{\dagger})^*\tilde{A}^{\dagger};$ 
    \item $AA^{\#}=A^{\#}A$ if and only if $A_{N,N}^{\dagger}(A_{N,N}^{\dagger})^{\#}=(A_{N,N}^{\dagger})^{\#}A_{N,N}^{\dagger}$, 
\end{enumerate}
where $A^*~\textnormal{and}~A^{\#}$ are the conjugate transpose and the weighted conjugate transpose of the matrix $A$, respectively and $\tilde{A}=N^{1/2}AN^{-1/2}.$
\end{corollary}
It is known that if $\lambda\neq 0$ is an eigenvalue of a normal tensor $\mc{A}$, then $1/ \lambda$ is an eigenvalue of its Moore-Penrose inverse $\mc{A}^{\dagger}$. However, this is not true in case of the weighted Moore-Penrose inverse $\mc{A}_{\mc{M},\mc{N}}^{\dagger}$ of $\mc{A}$. The next example is in this direction.
                     \begin{example}\label{exev}
Let $A=\begin{pmatrix}
1&1\\
1&1\\
\end{pmatrix}$. If $M=\begin{pmatrix}
1&0\\
0&2\\
\end{pmatrix}$ and $N=\begin{pmatrix}
3&0\\
0&1\\
\end{pmatrix}$ are two Hermitian positive definite matrices in $\mathbb{C}^{2\times2}$, then $A_{M,N}^{\dagger}=\begin{pmatrix}
\frac{1}{12}&\frac{1}{6}\\
\frac{1}{4}&\frac{1}{2}\\
\end{pmatrix}.$ Here $A$ is normal, and $2$ is an eigenvalue of $A$ but $1/2$ is not an eigenvalue of $A_{M,N}^{\dagger}$.
\end{example}
The weighted normal tensors fulfill the above requirement, i.e., if $\lambda\neq0$ is an eigenvalue of $\mc{A}$, then $1/ \lambda$ is an eigenvalue of its weighted Moore-Penrose inverse, in the case of weighted normal tensor. It is shown in the following theorem. 
                    \begin{theorem}\label{thmev}
Let $\mc{A}\in \mb{C}^{I_{1\ldots N}\times I_{1\ldots N}}$ be an even-order square tensor. If $\mc{N}\in \mb{C}^{I_{1\ldots N}\times I_{1\ldots N}}$ is a Hermitian positive definite tensor, then
                    \begin{enumerate}[(i)]
    \item  $\lambda\in \sigma(\mc{A})$ if and only if $\lambda\in \sigma(\tilde{\mc{A}})$;
    \item  $\lambda\in \sigma(\mc{A}_{\mc{N},\mc{N}}^{\dagger})$ if and only if $\lambda\in \sigma(\tilde{\mc{A}}^{\dagger})$;
    \item  if $\mc{A}$ is the weighted normal tensor 
    and $\lambda\neq0$, then $\lambda\in \sigma(\mc{A})$ if and only if $1/\lambda\in \sigma(\mc{A}_{\mc{N},\mc{N}}^{\dagger})$.
\end{enumerate}
\end{theorem}
Theorem \ref{thmev} reduces to Corollary \ref{cor:thmev1} in the matrix case.  
                   \begin{corollary}\label{cor:thmev1}
Let $A\in \mathbb{C}^{n\times n}$. If $N\in  \mathbb{C}^{n\times n}$ is a Hermitian positive definite matrix, then
                \begin{enumerate}[(i)]
    \item $\lambda\in \sigma(A)$ if and only if $\lambda\in \sigma(\tilde{A})$;
    \item $\lambda\in \sigma(A_{N,N}^{\dagger})$ if and only if $\lambda\in \sigma(\tilde{A}^{\dagger})$;
    \item if $A$ is the weighted normal matrix and $\lambda\neq0$, then $\lambda\in \sigma(A)$ if and only if $1/\lambda\in \sigma(A_{N,N}^{\dagger})$. 
\end{enumerate}
\end{corollary}
The next result provides a necessary and sufficient condition for a tensor to commute with its weighted Moore-Penrose inverse.
                \begin{theorem}\label{thm25}
Let $\mc{A}\in \mb{C}^{I_{1\ldots N}\times I_{1\ldots N}}$ be an even-order square tensor. If $\mc{N}\in \mb{C}^{I_{1\ldots N}\times I_{1\ldots N}}$ is a Hermitian positive definite tensor, then $ \mc{A}\n\mc{A}_{\mc{N},\mc{N}}^{\dagger}=\mc{A}_{\mc{N},\mc{N}}^{\dagger}\n\mc{A}~ \textnormal{if and only if }~\tilde{\mc{A}}\n\tilde {\mc{A}}^{\dagger}=\tilde {\mc{A}}^{\dagger}\n\tilde{\mc{A}}.$
\end{theorem}
The matrix version of the above result is given below.
                 \begin{corollary}\label{cor:thm25.1}
Let $A\in \mathbb{C}^{n\times n}$. If $N\in  \mathbb{C}^{n\times n}$ is a Hermitian positive definite matrix, then
$AA_{N,N}^{\dg}=A_{N,N}^{\dg}A$ if and only if $\tilde{A}\tilde{A}^{\dg}=\tilde{A}^{\dg}\tilde{A}$. 
\end{corollary}
Weighted normality is a sufficient condition for commutativity of a tensor with its weighted Moore-Penrose inverse. The following theorem is in this direction. 
                 \begin{theorem}\label{thm26}
Let $\mc{A}\in \mb{C}^{I_{1\ldots N}\times I_{1\ldots N}}$ be an even-order square tensor, and $\mc{N}\in \mb{C}^{I_{1\ldots N}\times I_{1\ldots N}}$ be a Hermitian positive definite tensor. If $\mc{A}$ is weighted normal, then $\mc{A}\n\mc{A}_{\mc{N},\mc{N}}^{\dagger}=\mc{A}_{\mc{N},\mc{N}}^{\dagger}\n\mc{A}.$
\end{theorem}
The next result conveys the matrix case of the above theorem. 
                  \begin{corollary}\label{cor:thm26.1}
Let $A\in \mathbb{C}^{n\times n}$, and $N\in \mathbb{C}^{n\times n}$ be a Hermitian positive definite matrix. If $A$ is weighted normal, then $A$ is a weighted EP-matrix w.r.t. $(N,N)$, i.e., $AA_{N,N}^{\dg}=A_{N,N}^{\dg}A.$
\end{corollary}
\section{Weighted tensor norm}\label{sec:wtnsrnorm}
For two Hermitian positive definite tensors $\mc{M}\in\mb{C}^{I_{1\ldots M}\times I_{1\ldots M}}$ and $\mc{N}\in\mb{C}^{J_{1\ldots N}\times J_{1\ldots N}}$, we define the weighted inner product and their induced weighted tensor norms here. The weighted inner products in $\mb{C}^{I_{1\ldots M}}$ and $\mb{C}^{J_{1\ldots N}}$ are  $$\left\langle\mc{X},\mc{Y}\right\rangle_\mc{M}=\left\langle\mc{M}\m\mc{X},\mc{Y}\right\rangle,~\mc{X},\mc{Y}\in \mb{C}^{I_{1\ldots M}}$$ and $$\left\langle\mc{X},\mc{Y}\right\rangle_\mc{N}=\left\langle\mc{N}\n\mc{X},\mc{Y}\right\rangle,~\mc{X},\mc{Y}\in \mb{C}^{J_{1\ldots N}}, $$ respectively.
Then, their induced weighted tensor norms are 
$$\|\mc{X}\|_\mc{M}=\sqrt{\left\langle\mc{X},\mc{X}\right\rangle_\mc{M}},~ \mc{X}\in\mb{C}^{I_{1\ldots M}}$$ and $$\|\mc{X}\|_\mc{N}=\sqrt{\left\langle\mc{X},\mc{X}\right\rangle_\mc{N}},~ \mc{X}\in\mb{C}^{J_{1\ldots N}},$$ respectively.  
                  \begin{lemma}\label{lem6}
For $\mc{W}\in \mathbb{C}^{I_{1\ldots M}\times J_{1\ldots N}}$, 
$\mc{X}\in\mb{C}^{J_{1\ldots N}}$,
and $\mc{Y}\in\mb{C}^{I_{1\ldots M}}$, we have $$\left\langle\mc{W}\n\mc{X},\mc{Y}\right\rangle=\left\langle\mc{X},\mc{W}^H\m\mc{Y}\right\rangle.$$
\end{lemma}
From Lemma \ref{lem6}, the next result is easy to deduce.
                 \begin{lemma}\label{lem7}
Let $ \mc{X}\in\mb{C}^{I_{1\ldots M}}$ and $ \mc{Y}\in\mb{C}^{J_{1\ldots N}}$. If $\mc{M}\in\mb{C}^{I_{1\ldots M}\times I_{1\ldots M}}$ and $\mc{N}\in\mb{C}^{J_{1\ldots N}\times J_{1\ldots N}}$ are two Hermitian positive definite tensors, then
                   \begin{enumerate}[(i)]
    \item $\|\mc{X}\|_\mc{M}=\|\mc{M}^{1/2}\m\mc{X}\|$;
    \item $\|\mc{X}\|_\mc{N}=\|\mc{N}^{1/2}\n\mc{X}\|$.
\end{enumerate}
\end{lemma}

Let $\mc{X},\mc{Y} \in\mb{C}^{I_{1\ldots M}}$ with a Hermitian positive definite tensor  $\mc{M}\in\mb{C}^{I_{1\ldots M}\times I_{1\ldots M}}$. Then, $\mc{X} ~\textnormal{and}~\mc{Y}$ are called $\mc{M}$-orthogonal if $\left\langle\mc{X},\mc{Y}\right\rangle_M=0$. Next, we prove the weighted Pythagorean theorem for tensors. 
                   
                    \begin{theorem}\label{thm:W Pyth tnsr}
Let $\mc{X}, \mc{Y} \in\mb{C}^{I_{1\ldots M}}$ be $\mc{M}$-orthogonal. Then, 
$$ \|\mc{X}+\mc{Y}\|_\mc{M}^2=\|\mc{X}\|_\mc{M}^2+\|\mc{Y}\|_\mc{M}^2.$$
\end{theorem}
For a tensor $\mc{A}\in \mb{C}^{I_{1\ldots M}\times J_{1\ldots N}}$, we define the tensor norm as: 
\begin{equation}\label{eqn:tnsr norm}
    \|\mc{A}\|=\textnormal{sup} \{\|\mc{A}\n\mc{X}\| ~:~    \|\mc{X}\|=1, ~\mc{X}\in\mb{C}^{J_{1\ldots N}} \}.
\end{equation}
By performing the steps of the existing proofs for matrices, we can verify the following.
\begin{enumerate}[(i)] 
    \item $\|\mc{A}\|=\|\mc{A}\|_2$, where $\|.\|_2$ is the spectral norm of $\mc{A}.$
    \item $\|\mc{A}\n\mc{X}\|\leq \|\mc{A}\| \|\mc{X}\|.$
    \item $\|\mc{A}\n\mc{B}\|\leq \|\mc{A}\| \|\mc{B}\|$, where $\mc{A}\in \mb{C}^{I_{1\ldots M}\times J_{1\ldots N}}$ and $\mc{B}\in \mb{C}^{J_{1\ldots N}\times K_{1\ldots L}}.$
    \item $\|\mc{A}\|=\|\mc{A}^H\|$.
    \item $\|\mc{A}^H\n\mc{A}\|=\|\mc{A}\|^2$.
\end{enumerate}

Now, for tensors $\mc{A}\in \mb{C}^{I_{1\ldots M}\times J_{1\ldots N}}$ and $\mc{B}\in \mb{C}^{J_{1\ldots N}\times I_{1\ldots M}}$ with two Hermitian positive definite tensors $\mc{M}\in\mb{C}^{I_{1\ldots M}\times I_{1\ldots M}}$ and $\mc{N}\in\mb{C}^{J_{1\ldots N}\times J_{1\ldots N}}$, we define the weighted tensor norms as:
                   \begin{equation}\label{eqn25}
                     \|\mc{A}\|_{\mc{M}\mc{N}}=\textnormal{sup} \{\|\mc{A}\n\mc{X}\|_\mc{M} ~:~    \|\mc{X}\|_\mc{N}=1, ~\mc{X}\in\mb{C}^{J_{1\ldots N}} \}  
\end{equation} 
and
                     \begin{equation}\label{eqn26}
                \|\mc{B}\|_{\mc{N}\mc{M}}=\textnormal{sup} \{\|\mc{B}\m\mc{X}\|_\mc{N} ~:~ \|\mc{X}\|_\mc{M}=1,~\mc{X}\in \mb{C}^{I_{1\ldots M}}\}.
\end{equation}
The following result provides a relation between the weighted tensor norm and the tensor norm.
                  \begin{lemma}\label{lem5}
Let $\mc{A}\in \mb{C}^{I_{1\ldots M}\times J_{1\ldots N}}$ and $\mc{B}\in \mb{C}^{J_{1\ldots N}\times I_{1\ldots M}}$. If $\mc{M}\in\mb{C}^{I_{1\ldots M}\times I_{1\ldots M}}$ and $\mc{N}\in\mb{C}^{J_{1\ldots N}\times J_{1\ldots N}}$ are two Hermitian positive definite tensors, then
                 \begin{enumerate}[(i)]
     \item $\|\mc{A}\|_{\mc{M}\mc{N}}= \|\mc{M}^{1/2}\m\mc{A}\n\mc{N}^{-1/2}\|;$
    \item $ \|\mc{B}\|_{\mc{N}\mc{M}}= \|\mc{N}^{1/2}\n\mc{B}\m\mc{M}^{-1/2}\|.$ 
\end{enumerate}
\end{lemma}
The following lemma shows the consistent property of the weighted tensor norm.
            \begin{lemma}\label{lem9}
Let $\mc{A}\in \mb{C}^{I_{1\ldots M}\times J_{1\ldots N}}$, $\mc{B}\in \mb{C}^{J_{1\ldots N}\times I_{1\ldots M}}$, $\mc{X} \in\mb{C}^{J_{1\ldots N}}$, and $\mc{Y} \in\mb{C}^{I_{1\ldots M}}$. If $\mc{M}\in\mb{C}^{I_{1\ldots M}\times I_{1\ldots M}}$ and $\mc{N}\in\mb{C}^{J_{1\ldots N}\times J_{1\ldots N}}$ are two Hermitian positive definite tensors, then
                  \begin{enumerate}[(i)]
    \item $\|\mc{A}\n\mc{X}\|_\mc{M}\leq\|\mc{A}\|_{\mc{M}\mc{N}} ~\|\mc{X}\|_N;$
    \item $\|\mc{B}\m\mc{Y}\|_\mc{N}\leq\|\mc{B}\|_{\mc{N}\mc{M}} ~\|\mc{Y}\|_\mc{M};$
    \item $\|\mc{A}\n\mc{B}\|_{\mc{M}\mc{M}}\leq\|\mc{A}\|_{\mc{M}\mc{N}}~\|\mc{B}\|_{\mc{N}\mc{M}}.$
\end{enumerate}
\end{lemma}
The following lemma comprises some properties of the weighted conjugate
transpose with the weighted tensor norm.
                    \begin{lemma}\label{lem10}
Let $\mc{A}\in \mb{C}^{I_{1\ldots M}\times J_{1\ldots N}}$. If $\mc{M}\in\mb{C}^{I_{1\ldots M}\times I_{1\ldots M}}$ and $\mc{N}\in\mb{C}^{J_{1\ldots N}\times J_{1\ldots N}}$ are two Hermitian positive definite tensors, then 
                   \begin{enumerate}[(i)]
    \item $\|\mc{A}\|_{\mc{M}\mc{N}}=\|\mc{A}^{\#}_{\mc{M}\mc{N}}\|_{\mc{N}\mc{M}}$;
    \item $\|\mc{A}\|_{\mc{M}\mc{N}}^2=\|\mc{A}\n\mc{A}^{\#}_{\mc{M}\mc{N}}\|_{\mc{M}\mc{M}}=\|\mc{A}^{\#}_{\mc{M}\mc{N}}\m\mc{A}\|_{\mc{N}\mc{N}}.$
\end{enumerate}
\end{lemma}
The next result defines the weighted tensor norm as the maximum $(\mc{M},\mc{N})$ singular value of $\mc{A}.$ 
                 \begin{theorem}\label{Wnorm mu1}
Let $\mc{A}\in \mb{C}^{I_{1\ldots M}\times J_{1\ldots N}}$. If $\mc{M}\in \mb{C}^{I_{1\ldots M}\times I_{1\ldots M}}$ and $\mc{N}\in \mb{C}^{J_{1\ldots N}\times J_{1\ldots N}}$ are two Hermitian positive definite tensors, then 
$$\|\mc{A}\|_{\mc{M}\mc{N}}=\mu_{max} ~\textnormal{and}~\|\mc{A}_{\mc{M},\mc{N}}^{\dagger}\|_{\mc{N}\mc{M}}=\frac{1}{\mu_{min}},$$ where $\mu_{max}~\textnormal{and}~\mu_{min}$ are the maximum and minimum $(\mc{M},\mc{N})$ singular values of $\mc{A}.$ 
\end{theorem}

\section{Numerical range for the weighted Moore-Penrose inverse of an even-order square tensor}\label{sec:nrwmpi}  
In this section, we establish several properties of the numerical ranges of a tensor and its weighted Moore-Penrose inverse.
The first result conveys that the spectra of $\mc{A}$ and $\mc{A}_{\mc{M}, \mc{N}}^{\dg}$ as well as their numerical ranges simultaneously contain the origin.
                   \begin{theorem}\label{thm1}
Let $\mc{A}\in \mb{C}^{I_{1\ldots N}\times I_{1\ldots N}}$. If $\mc{M}, \mc{N} \in \mb{C}^{I_{1\ldots N}\times I_{1\ldots N}}$ are two Hermitian positive definite tensors, then 
                   \begin{enumerate}[(i)]
        \item $0\in\sigma{(\mc{A})}$ if and only if $0\in\sigma(\mc{A}_{\mc{M}, \mc{N}}^{\dg})$;
        \item $0\in W(\mc{A})$ if and only if $0\in W(\mc{A}_{\mc{M}, \mc{N}}^{\dg})$.
\end{enumerate}
\end{theorem}
An immediate consequence of the above result when the tensor weights $\mc{M}$ and $\mc{N}$ are considered to be identity tensors is as follows.
                 \begin{corollary}[Theorems 5.2 and 5.4, \cite{nirmal}]\label{cor:thm1.1}\leavevmode\\
 Let $\mc{A}\in \mb{C}^{I_{1\ldots N}\times I_{1\ldots N}}$. Then, 
                \begin{enumerate}[(i)]
        \item $0\in\sigma{(\mc{A})}$ if and only if $0\in\sigma(\mc{A}^{\dg})$;
        \item $0\in W(\mc{A})$ if and only if $0\in W(\mc{A}^{\dg})$.
    \end{enumerate}  
  \end{corollary}
   We have the following corollary in the matrix setting that generalizes Theorem 2 in \cite{chien2020}. 
                  \begin{corollary}\label{cor:thm1.2}
  Let $\mc{A}\in \mb{C}^{n\times n}$. If $M,N\in\mb{C}^{n\times n} $ are two Hermitian positive definite matrices, then 
                  \begin{enumerate}[(i)]
        \item $0\in\sigma{(A)}$ if and only if $0\in\sigma({A}_{M,N}^{\dg})$;
        \item $0\in W(A)$ if and only if $0\in W(A_{M,N}^{\dg})$.
    \end{enumerate}   
   \end{corollary}
As an application of Theorem \ref{thm1}, we have the following result.
                   \begin{theorem}\label{thm2} Let $\{z_{i_1i_2\cdots i_N}\}_{i_j=1}^{I_j}, ~\textnormal{for}~ j=1,2,..., N$ be nonzero complex numbers. If $$0=\sum_{i_1,i_2,...,i_N}\alpha_{i_1i_2\cdots i_N}z_{i_1i_2\cdots i_N},$$ for some nonnegative scalars $\{\alpha_{i_1i_2\cdots i_N}\}_{i_j=1}^{I_j}  ~\textnormal{for}~ j=1,2,..., N$ with $\displaystyle \sum_{i_1,i_2,...,i_N}\alpha_{i_1i_2\cdots i_N}=1$, then there exist nonnegative scalars $\{\beta_{i_1i_2\cdots i_N}\}_{i_j=1}^{I_j}  ~\textnormal{for}~ j=1,2,..., N$ with $\displaystyle\sum_{i_1,i_2,...,i_N}\beta_{i_1i_2\cdots i_N}=1$ such that $$0=\sum_{i_1,i_2,...,i_N}\beta_{i_1i_2\cdots i_N}\frac{1}{z_{i_1i_2\cdots i_N}}.$$ 
\end{theorem}

In like manner it is easy to show Corollary \ref{cor:thm2.1} for the matrix case.                     
                   \begin{corollary}[Theorem 3, \cite{chien2020}]\label{cor:thm2.1}\leavevmode\\
Let $z_1, z_2, \ldots, z_n$ be nonzero complex numbers. If $$0=\alpha_1z_1+\alpha_2z_2+\ldots+\alpha_nz_n$$ for some nonnegative scalars $\alpha_1, \alpha_2, \ldots, \alpha_n$ with $\alpha_1+ \alpha_2+ \cdots+ \alpha_n=1$, then there exist nonnegative scalars  $\beta_1,\beta_2,\ldots,\beta_n$ with $\beta_1+\beta_2+\cdots+\beta_n=1$ such that $$0=\beta_1\frac{1}{z_1}+\beta_2\frac{1}{z_2}+\cdots+\beta_n\frac{1}{z_n}.$$
\end{corollary}  
 Next theorem establishes a relation among $\sigma(\mc{A})$, $W(\mc{A})$, and $\dfrac{1}{W(\mc{A}_{\mc{M}, \mc{N}}^{\dg})}$.
                      \begin{theorem}\label{thm3}
Let $\mc{A}\in \mb{C}^{I_{1\ldots N}\times I_{1\ldots N}}$, and $\mc{M},\mc{N} \in \mb{C}^{I_{1\ldots N}\times I_{1\ldots N}}$ be two Hermitian positive definite tensors. If $\mc{A}\n \mc{A}_{\mc{M}, \mc{N}}^{\dg}=\mc{A}_{\mc{M}, \mc{N}}^{\dg}\n \mc{A}$, then 
                      \begin{equation}\label{eqn1}
    \sigma(\mc{A})\subset W(\mc{A})\bigcap \dfrac{1}{W(\mc{A}_{\mc{M}, \mc{N}}^{\dg})}.
\end{equation}
\end{theorem}
Our next example shows that the assumption $\mc{A}\n \mc{A}_{\mc{M}, \mc{N}}^{\dg}=\mc{A}_{\mc{M}, \mc{N}}^{\dg}\n \mc{A}$ in Theorem \ref{thm3} is essential.
                         \begin{example}\label{exm1}
Let $A=\begin{pmatrix}
1&1\\
0&0\\
\end{pmatrix}\in \mathbb{C}^{2\times2}$. If $M=\begin{pmatrix}
1&0\\
0&2\\
\end{pmatrix}$, $N=\begin{pmatrix}
3&0\\
0&1\\
\end{pmatrix}$ are two Hermitian positive definite matrices in $\mathbb{C}^{2\times2}$, then $A_{M,N}^{\dagger}=\begin{pmatrix}
\frac{1}{4}&0\\
\frac{3}{4}&0\\
\end{pmatrix}$,
$AA_{M,N}^{\dagger}=\begin{pmatrix}
1&0\\
0&0\\
\end{pmatrix},~ \textnormal{and}~
 A_{M,N}^{\dagger}A=\begin{pmatrix}
\frac{1}{4}&\frac{1}{4}\\
\frac{3}{4}&\frac{3}{4}\\
\end{pmatrix}$.
So, $AA_{M,N}^{\dagger}\neq A_{M,N}^{\dagger}A$.
Now,
                      \begin{eqnarray*}
W(A_{M,N}^{\dagger})&=&\left\{\frac{1}{4}\bar z_{1}z_{1}+\frac{3}{4}\bar z_{2}z_{1}:~z=(z_1,~z_2)^T \in \mathbb{C}^2,~|z_1|^2+|z_2|^2=1\right\},\\
\frac{1}{W(A_{M,N}^{\dagger})}&=&\left\{\frac{1}{z}:~z\in W(A_{M,N}^{\dagger})\right\}. 
\end{eqnarray*}
Let $\alpha \in W(A_{M,N}^{\dagger}). $ Then,
                     \begin{eqnarray*}
\alpha&=&\frac{1}{4}\bar z_{1}z_{1}+\frac{3}{4}\bar z_{2}z_{1},~~\textnormal{for some}~z=(z_1,~z_2)^T \in \mathbb{C}^2,~|z_1|^2+|z_2|^2=1,\\
|\alpha|&\leq& \frac{1}{4}|z_1|^2+\frac{3}{4}|z_2||z_1| \\
&\leq& \frac{1}{4}(|z_1|^2+|z_2|^2)+\frac{3}{4}\left(\frac{|z_1|^2+|z_2|^2}{2}\right)\\
&=& \frac{1}{4}\times 1+\frac{3}{8}\times 1\\
&=&\frac{5}{8}\\
&<&1.
\end{eqnarray*}
Let $\beta \in \frac{1}{W(A_{M,N}^{\dagger})}$. Then, $\beta = \frac{1}{\alpha}~\textnormal{for some}~\alpha \in W(A_{M,N}^{\dagger}),$ therefore, $|\beta|>1.$ Thus, the eigenvalue $1$ of $A$ is not in $\frac{1}{W(A_{M,N}^{\dagger})}.$
\end{example}
If $\mc{M}$ and $\mc{N}$ both are identity tensors, then the above result coincides with the following corollary.
                      \begin{corollary}[Theorem 5.7, \cite{nirmal}]\label{cor:thm3.1}\leavevmode\\
Let $\mc{A}\in \mb{C}^{I_{1\ldots N}\times I_{1\ldots N}}$. If $\mc{A}\n \mc{A}^{\dg}=\mc{A}^{\dg}\n \mc{A}$, then
                     \begin{equation*}
    \sigma(\mc{A})\subset W(\mc{A})\bigcap \dfrac{1}{W(\mc{A}^{\dg})}.
\end{equation*} 
\end{corollary}
The following result generalizes Theorem 5 in
\cite{chien2020}.
                     \begin{corollary}\label{cor:thm3.2}
Let ${A}\in \mb{C}^{n\times n}$ be a square matrix, and $M,N\in\mb{C}^{n\times n} $ be two Hermitian positive definite matrices. If $A$ is weighted EP-matrix, i.e., $AA_{M,N}^{\dg}=A_{M,N}^{\dg}A$, then \begin{equation*}
    \sigma(A)\subset W(A)\bigcap \dfrac{1}{W({A}_{M,N}^{\dg})}.
\end{equation*} 
\end{corollary}  
The next theorem is an application of the WSVD.
                     \begin{theorem}\label{thm:intrsn MNsing val tnsr}
Let $\mc{A}\in \mb{C}^{I_{1\ldots N}\times I_{1\ldots N}}$ such that $W(\mc{A})=W(\mc{A}^H)$. If $\mc{M}=\mc{N}=\beta \mc{I}\in \mb{C}^{I_{1\ldots N}\times I_{1\ldots N}}$ are two Hermitian positive definite tensors, then $$ W(\mc{A})\bigcap \mu ^2 W(\mc{A}_{\mc{M},\mc{N}}^{\dagger})\neq \emptyset,$$ for every $(\mc{M},\mc{N})$ singular value $\mu$ of $\mc{A}.$
\end{theorem}
When  $\beta=1$, we have the following result. 
                      \begin{corollary}[Theorem 5.5, \cite{nirmal}]\label{cor:intrsn sing val tnsr}\leavevmode\\
    Let $\mc{A}\in \mb{C}^{I_{1\ldots N}\times I_{1\ldots N}}$ such that $W(\mc{A})=W(\mc{A}^{H})$. Then, 
                     \begin{equation*}
        W(\mc{A})\bigcap \alpha^{2}W(\mc{A}^{\dg})\neq \emptyset,
\end{equation*}
where $\alpha$ is any singular value of $\mc{A}$.\end{corollary}
It can be verified that Corollary \ref{cor:intrsn MNsing val mtrx} generalizes Theorem 4 in \cite{chien2020}. 
                        \begin{corollary}\label{cor:intrsn MNsing val mtrx}
                        
Let $A\in\mb{C}^{n\times n}$ such that $W(A)$ is symmetric with respect to $X$-axis, i.e., $W({A})=W({A}^*)$. If $M=N= \beta I\in\mb{C}^{n\times n}$ are two Hermitian positive definite matrices, then $$ W({A})\bigcap \mu ^2 W({A}_{{M},{N}}^{\dagger})\neq \emptyset,$$ for every $({M},{N})$ singular value $\mu$ of ${A}.$  
\end{corollary}  
In the following result, we derive that the numerical ranges for the weighted Moore-Penrose inverse and the weighted conjugate transpose are equal for a particular type of tensor.
                 \begin{theorem}\label{thm6}
Let $\mc{A}\in \mb{C}^{I_{1\ldots N}\times I_{1\ldots N}}$, and $\mc{M}, \mc{N} \in \mb{C}^{I_{1\ldots N}\times I_{1\ldots N}}$ be two Hermitian positive definite tensors. Let $\{\mc{U}_{1},~\mc{U}_{2},~\ldots,~\mc{U}_{r}\}$ be an $\mc{M}$-orthonormal and $\{\mc{V}_{1},~\mc{V}_{2},~\ldots,~\mc{V}_{r}\}$ be an $\mc{N}^{-1}$-orthonormal subsets of $\mb{C}^{I_{1\ldots N}}$. If $\mc{A}=\mc{U}_{1}\1\mc{V}_{1}^{H}+\mc{U}_{2}\1\mc{V}_{2}^{H}+\cdots+\mc{U}_{r}\1\mc{V}_{r}^{H}$, then $\mc{A}_{\mc{M},\mc{N}}^{\dg}=\mc{N}^{-1}\n(\mc{V}_{1}\1\mc{U}_{1}^{H}+\mc{V}_{2}\1\mc{U}_{2}^{H}+\cdots+\mc{V}_{r}\1\mc{U}_{r}^{H})\n \mc{M}$ and $W(\mc{A}_{\mc{M},\mc{N}}^{\dg})=W(\mc{A}_{\mc{M}\mc{N}}^{\#})$.   
\end{theorem}
The above result reduces to the following corollary when we consider identity tensor as weights. 
                \begin{corollary}\label{cor:thm6.1}
Let $\mc{A}\in \mb{C}^{I_{1\ldots N}\times I_{1\ldots N}}$.
Let $\{\mc{U}_{1},~\mc{U}_{2},~\ldots,~\mc{U}_{r}\}$ and $\{\mc{V}_{1},~\mc{V}_{2},~\ldots,~\mc{V}_{r}\}$ be two orthonormal subsets of $\mb{C}^{I_{1\ldots N}}$. If $\mc{A}=\mc{U}_{1}\1\mc{V}_{1}^{H}+\mc{U}_{2}\1\mc{V}_{2}^{H}+\cdots+\mc{U}_{r}\1\mc{V}_{r}^{H}$, then $\mc{A}^{\dg}=\mc{V}_{1}\1\mc{U}_{1}^{H}+\mc{V}_{2}\1\mc{U}_{2}^{H}+\cdots+\mc{V}_{r}\1\mc{U}_{r}^{H}$ and $W(\mc{A}^{\dg})=W(\mc{A}^{H})$.
\end{corollary}
The following corollary is a generalization of Theorem 6 in \cite{chien2020} and the proof is different than the above proof.
                     \begin{corollary}\label{cor:thm6.2}
Let ${A}\in \mb{C}^{n\times n}$, and $M,N\in \mb{C}^{n\times n}$ be two Hermitian positive definite matrices.  Let $\{u_1,u_2,\ldots,u_r\}$ be a subset of $M$-orthonormal vectors and $\{v_1,v_2,\ldots,v_r\}$ be a subset of $N^{-1}$-orthonormal vectors of $\mathbb{C}^n$, respectively. If $A=u_1v_1^*+u_2v_2^*+\cdots+u_rv_r^*$, then $A_{M,N}^{\dagger}=N^{-1}(v_1^*u_1+v_2^*u_2+\cdots+v_r^*u_r)M$ and $W(A_{M,N}^{\dagger})=W(A^{\#}),$ where $A^{\#}$ is the weighted conjugate transpose of the matrix $A.$ 
\end{corollary}
As an application of the weighted tensor norm, the following result provides a bound for the product of the weighted tensor norms of a tensor $\mc{A}$ and its weighted Moore-Penrose inverse $\mc{A}_{\mc{M},\mc{N}}^{\dg}$ in terms of the product of numerical radii of the tensor $\mc{\tilde{A}}$ and its Moore-Penrose inverse $\mc{\tilde{A}}^{\dg}$.
                     \begin{theorem}\label{thmnormineq}
Let $\mc{O}\neq \mc{A}\in \mb{C}^{I_{1\ldots N}\times I_{1\ldots N}}$. If $\mc{M},\mc{N}\in\mb{C}^{I_{1\ldots N}\times I_{1\ldots N}}$ are two Hermitian positive definite tensors, then for the weighted tensor norm $\|\mc{A}\|_{\mc{M}\mc{N}}$,
                    \begin{equation*}
    1\leq \|\mc{A}\|_{\mc{M}\mc{N}}\|\mc{A}_{\mc{M},\mc{N}}^{\dg}\|_{\mc{N}\mc{M}}\leq 4w(\mc{\tilde{A}})w(\mc{\tilde{A}}^{\dg}),
\end{equation*}
where $\mc{\tilde{A}}=\mc{M}^{1/2}\n\mc{A}\n\mc{N}^{-1/2}$ and $\|.\|$ is the spectral norm of a tensor.
\end{theorem}
In particular, the identity weights give the following corollary.
                    \begin{corollary}[Theorem 5.9, \cite{nirmal}]\label{cor:thmnormineq1}\leavevmode\\
Let $\mc{O}\neq \mc{A}\in \mb{C}^{I_{1\ldots N}\times I_{1\ldots N}}$. Then, for the spectral norm $\|\cdot\|$,
                      \begin{equation*}
    1\leq \|\mc{A}\|\|\mc{A}^{\dg}\|\leq 4w(\mc{A})w(\mc{A}^{\dg}).
\end{equation*}
\end{corollary}
The next corollary is a generalization of Theorem 7 in \cite{chien2020}. 
                      \begin{corollary}\label{cor:thmnormineq2}
Let $0\neq A\in\mb{C}^{n\times n}$, and $M, N \in  \mb{C}^{n\times n}$ be two Hermitian positive definite matrices. Then, for the weighted matrix norm $\|A\|_{MN}=\|M^{1/2}AN^{-1/2}\|$, 
$$1 \leq \|A\|_{MN}\|A_{M,N}^{\dagger}\|_{NM} \leq 4\omega(\tilde{A})\omega (\tilde{A}^{\dagger}),$$
where $\tilde{A}=M^{1/2}AN^{-1/2}$, $A_{M,N}^{\dagger}=N^{-1/2}\tilde{A}^{\dagger}M^{1/2}$ and $\|.\|$ is the spectral norm.
\end{corollary} 
With respect to the diagonal weights, the weighted Moore-Penrose inverse of a weighted shift matrix is again a weighted shift matrix; this is shown in the following theorem. Also, for their numerical radii, some upper bounds are established.
                     \begin{theorem}\label{th7}
Let $A\in \mathbb{C}^{n\times n}$ be a weighted shift matrix                                            \begin{equation}\label{eq10}
A=\begin{pmatrix}
0&a_1&0&\cdots&0\\
0&0&a_2&\cdots&0\\
\vdots&\vdots&\ddots&\ddots&\vdots\\
0&0&0&\ddots&a_{n-1}\\
0&0&0&\cdots&0
\end{pmatrix}.
\end{equation}
If $M,N\in \mathbb{C}^{n\times n}$ are two positive diagonal matrices, then 
                       \begin{equation}\label{eq11}
A_{M,N}^{\dagger}= \begin{pmatrix}
0&0&\cdots&0&0\\
1/a_1&0&\cdots&0&0\\
0&1/a_2&\cdots&0&0\\
\vdots&\vdots&\ddots&\vdots&\vdots\\
0&0&\cdots&1/a_{n-1}&0\\
\end{pmatrix}.
\end{equation} 
Furthermore,
                      \begin{enumerate}[(i)]
    \item$W(A)$, $W(A_{M,N}^{\dagger})$ are circular disks centered at the origin, and $$\omega({A})~ \omega ({A}_{M,N}^{\dagger})\leq \frac{\textnormal{max}|a_k|}{\textnormal{min}|a_k|}\cos^2\left(\frac{\pi}{n+1}\right),$$ where minimum is taken over those $k$ with $a_k\neq0.$
    \item If $a_ka_{n-k}=1$ for all $k=1,2,\hdots,[n/2]$, then $W(A)=W(A_{M,N}^{\dagger})$, and $$\omega({A}) =\omega ({A}_{M,N}^{\dagger})\leq \textnormal{max}\{|a_k|, 1/|a_k|\}\cos\left(\frac{\pi}{n+1}\right).$$  
    \end{enumerate}
\end{theorem}
Instead of diagonal matrices, if we take $M$ and $N$ as identity matrices, then Theorem \ref{th7} coincides with the next result.  
                       \begin{corollary}[Theorem 8, \cite{chien2020}]\label{cor:th7.1}\leavevmode\\ 
Let $A\in \mathbb{C}^{n\times n}$ be a weighted shift matrix                
                      \begin{equation*}
A=\begin{pmatrix}
0&a_1&0&\cdots&0\\
0&0&a_2&\cdots&0\\
\vdots&\vdots&\ddots&\ddots&\vdots\\
0&0&0&\ddots&a_{n-1}\\
0&0&0&\cdots&0
\end{pmatrix}.
\end{equation*}
Then, 
                     \begin{equation*}
A^{\dagger}= \begin{pmatrix}
0&0&\cdots&0&0\\
1/a_1&0&\cdots&0&0\\
0&1/a_2&\cdots&0&0\\
\vdots&\vdots&\ddots&\vdots&\vdots\\
0&0&\cdots&1/a_{n-1}&0\\
\end{pmatrix}.
\end{equation*}
Furthermore,
                    \begin{enumerate}[(i)]
    \item $W(A)$, $W(A^{\dagger})$ are circular disks centered at the origin, and $$1\leq \omega({A})~ \omega ({A}^{\dagger})\leq \frac{\textnormal{max}|a_k|}{\textnormal{min}|a_k|}\cos^2\left(\frac{\pi}{n+1}\right),$$ where minimum is taken over those $k$ with $a_k\neq0.$
    \item If $a_ka_{n-k}=1$ for all $k=1,2,\hdots,[n/2]$, then $W(A)=W(A^{\dagger})$, and $$\frac{1}{2}\leq \omega({A}) =\omega ({A}^{\dagger})\leq \textnormal{max}\{|a_k|, 1/|a_k|\}\cos\left(\frac{\pi}{n+1}\right).$$  
\end{enumerate} 
\end{corollary}
We end this section with an example in which we plot the numerical ranges of a tensor, and its Moore-Penrose inverse and weighted Moore-Penrose inverse using the Algorithm 1 of \cite{nirmal}. To compute the Moore-Penrose inverse and the weighted Moore-Penrose inverse, we apply Algorithms \ref{alg1} and \ref{alg2} here. 
                    \begin{example}
Consider $\mc{A}\in\mb{C}^{2\times3\times 2\times 3}$ and the two weights $\mc{M},\mc{N}$ in  $\mb{C}^{2\times3\times 2\times 3}$ such that
                    \begin{center}
{\begin{tabular}{ccc|ccc|ccc|ccc|ccc|ccc}
\hline
\multicolumn{3}{c}{$\mc{A}(:,:,1,1)$} & \multicolumn{3}{c}{$\mc{A}(:,:,2,1)$} & \multicolumn{3}{c}{$\mc{A}(:,:,1,2)$} & \multicolumn{3}{c}{$\mc{A}(:,:,2,2)$} &
\multicolumn{3}{c}{$\mc{A}(:,:,1,3)$} &
\multicolumn{3}{c}{$\mc{A}(:,:,2,3)$} \\
\hline
    1 & 1 & 2 & 1 & 1 & 1 & 2 & 2 & 1 & 3 & 3 & 2 & 1 & 1 & 2 & 3 & 3 & 3\\
    1 & 1 & 2 & 2 & 2 & 1 & 2 & 2 & 1 & 4 & 4 & 2 & 1 & 1 & 2 & 3 & 3 & 3   \\
\hline
\end{tabular}},
\end{center}
                      \begin{center}
{\begin{tabular}{ccc|ccc|ccc|ccc|ccc|ccc}
\hline
\multicolumn{3}{c}{$\mc{M}(:,:,1,1)$} & \multicolumn{3}{c}{$\mc{M}(:,:,2,1)$} & \multicolumn{3}{c}{$\mc{M}(:,:,1,2)$} & \multicolumn{3}{c}{$\mc{M}(:,:,2,2)$} &
\multicolumn{3}{c}{$\mc{M}(:,:,1,3)$} &
\multicolumn{3}{c}{$\mc{M}(:,:,2,3)$} \\
\hline
     1 & 0 & 0 & 0 & 0 & 0 & 0 & 3 & 0 & 0 & 0 & 0 & 0 & 0 & 2 & 0 & 0 & 0\\
     0 & 0 & 0 & 2 & 0 & 0 & 0 & 0 & 0 & 0 & 1 & 0 & 0 & 0 & 0 & 0 & 0 & 3   \\
\hline
\end{tabular}},
\end{center}
and 
                  \begin{center}
{\begin{tabular}{ccc|ccc|ccc|ccc|ccc|ccc}
\hline
\multicolumn{3}{c}{$\mc{N}(:,:,1,1)$} & \multicolumn{3}{c}{$\mc{N}(:,:,2,1)$} & \multicolumn{3}{c}{$\mc{N}(:,:,1,2)$} & \multicolumn{3}{c}{$\mc{N}(:,:,2,2)$} &
\multicolumn{3}{c}{$\mc{N}(:,:,1,3)$} &
\multicolumn{3}{c}{$\mc{N}(:,:,2,3)$} \\
\hline
    3 & 0 & 0 & 0 & 0 & 0 & 0 & 1 & 0 & 0 & 0 & 0 & 0 & 0 & 1 & 0 & 0 & 0\\
    0 & 0 & 0 & 2 & 0 & 0 & 0 & 0 & 0 & 0 & 1 & 0 & 0 & 0 & 0 & 0 & 0 & 1   \\
\hline
\end{tabular}}.
\end{center}
By Algorithms \ref{alg1} and \ref{alg2}, the Moore-Penrose inverse and the weighted Moore-Penrose inverse of $\mc{A}$ are given as 
                   \begin{center}
{\begin{tabular}{ccc|ccc|ccc}
\hline
\multicolumn{3}{c}{$\mc{A}^{\dg}(:,:,1,1)$} & \multicolumn{3}{c}{$\mc{A}^{\dg}(:,:,2,1)$} & \multicolumn{3}{c}{$\mc{A}^{\dg}(:,:,1,2)$}\\
\hline
     -3/26 & 9/26 & -3/26 & 0 & -1/6 & 0 & -3/26 & 9/26 & -3/26\\
     -11/26 & -1/13 & 3/13 & 1/3 & 1/6 & -1/6 & -11/26 & -1/13 & 3/13\\
\hline
\end{tabular}}
\end{center}
                     \begin{center}
{\begin{tabular}{ccc|ccc|ccc}
\hline
 \multicolumn{3}{c}{$\mc{A}^{\dg}(:,:,2,2)$} &
\multicolumn{3}{c}{$\mc{A}^{\dg}(:,:,1,3)$} &
\multicolumn{3}{c}{$\mc{A}^{\dg}(:,:,2,3)$} \\
\hline
     0 & -1/6 & 0 & 2/13 & -5/39 & 2/13 & 2/13 & -5/39 & 2/13\\
     1/3 & 1/6 & -1/6 & 5/78 & -5/78 & 1/39 & 5/78 & -5/78 & 1/39   \\
\hline
\end{tabular}}
\end{center}
and 
                      \begin{center}
{\begin{tabular}{ccc|ccc|ccc}
\hline
\multicolumn{3}{c}{$\mc{A}_{\mc{M},\mc{N}}^{\dg}(:,:,1,1)$} & \multicolumn{3}{c}{$\mc{A}_{\mc{M},\mc{N}}^{\dg}(:,:,2,1)$} & \multicolumn{3}{c}{$\mc{A}_{\mc{M},\mc{N}}^{\dg}(:,:,1,2)$} \\
\hline
    -1/32 & 7/32 & -3/32 & 1/90 & -3/10 & 1/30 & -3/32 & 21/32 & -9/32 \\
    -5/32 & -3/32 & 1/8 & 29/90 & 31/90 & -4/15 & -15/32 & -9/32 & 3/8    \\
\hline
\end{tabular}}
\end{center}
                      \begin{center}
{\begin{tabular}{ccc|ccc|ccc}
\hline
 \multicolumn{3}{c}{$\mc{A}_{\mc{M},\mc{N}}^{\dg}(:,:,2,2)$} &
\multicolumn{3}{c}{$\mc{A}_{\mc{M},\mc{N}}^{\dg}(:,:,1,3)$} &
\multicolumn{3}{c}{$\mc{A}_{\mc{M},\mc{N}}^{\dg}(:,:,2,3)$} \\
\hline
     1/180 & -3/20 & 1/60 & 17/300 & -13/100 & 17/100 & 17/200 & -39/200 & 51/200\\
    29/180 & 31/180 & -2/15 & 13/300 & -13/300 & 1/25 & 13/200 & -13/200 & 3/50   \\
\hline
\end{tabular}},
\end{center} respectively. Now, applying Algorithm 1 of \cite{nirmal} to the tensors $\mc{A}$, $\mc{A}^{\dg}$, and $\mc{A}_{\mc{M},\mc{N}}^{\dg}$ for 500 different choices of $\theta$, we obtain Figure \ref{fig:nrfig1},
and the colored doted points inside the plotted region represent the eigenvalues of the corresponding tensor. 
                    \begin{figure}[h!]
    \centering
    \includegraphics[scale=0.5]{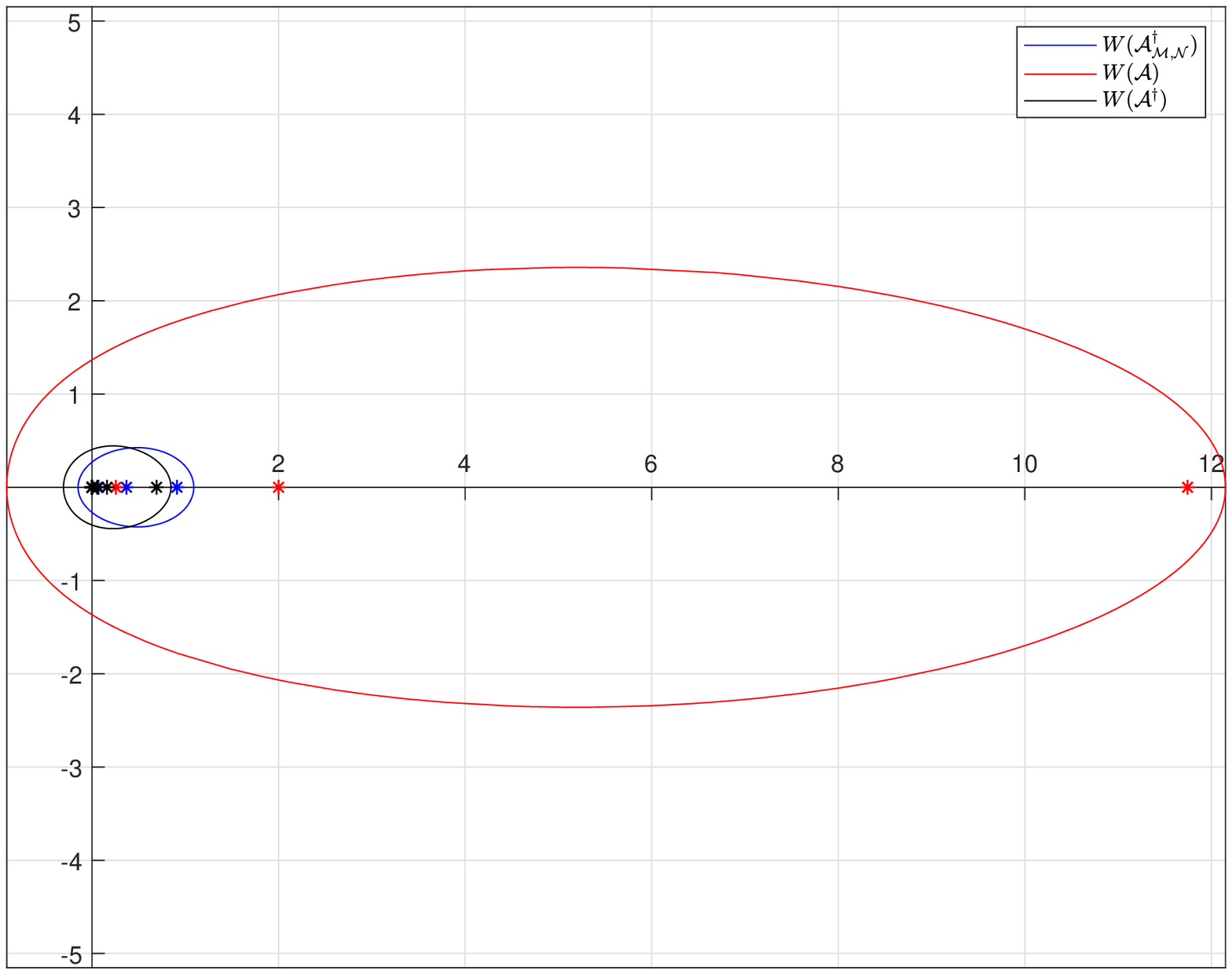}
    \caption{Numerical ranges of the tensors $\mc{A}$, $\mc{A}^{\dagger}$, and $\mc{A}_{\mc{M},\mc{N}}^{\dg}$}
    \label{fig:nrfig1}
\end{figure}

\end{example}
\newpage
\section{Conclusions}\label{sec:conclusion}
In this article, we have introduced the notion of the WSVD and derived the formula for computing the weighted Moore-Penrose inverse of an arbitrary-order tensor using the WSVD. After that, we have defined the notions of weighted normal tensor and weighted tensor norm. Further, we have established several properties that examine some relationship between a tensor's numerical range and its weighted Moore-Penrose inverse. An upper bound for the product of the numerical radii of a weighted shift matrix and its weighted Moore-Penrose inverse with diagonal weights has been established. An equality between the numerical ranges of the weighted Moore-Penrose inverse and the weighted conjugate transpose of a special tensor has been given. Our work on numerical ranges and numerical radii will also be beneficial in finding the iterative solution to tensor equations.
These theories add new contributions to the theory of tensors and will be crucial for future research on tensors.\\
\section*{Acknowledgements}
The first author acknowledges the support of the Council of Scientific and Industrial Research, India. We thank Dr. Krushnachandra Panigrahy for his insightful suggestions and discussions.
\section*{Conflict of Interest}
 The authors declare that there is no conflict of interest.

\section*{Data Availability Statement}
Data sharing is not applicable to this article as no new data is analyzed in this study.

\bibliographystyle{amsplain}
             
\end{document}